\documentclass[12pt,reqno]{amsart}
\usepackage{xcolor}
\usepackage{eucal}
\usepackage{amsthm}
\usepackage{amsfonts}
\usepackage{amsmath}
\usepackage{amssymb}
\usepackage{mathrsfs}
\usepackage{epsfig}
\usepackage{leftindex}
\usepackage{hyperref}
\usepackage{float}
\usepackage{MnSymbol} 
\newcommand{\ncom}{\newcommand}

\ncom{\ul}{\underline}
\ncom{\ol}{\overline}
\ncom{\bq}{\begin{equation}}
\ncom{\eq}{\end{equation}}
\ncom{\beqn}{\begin{eqnarray*}}
\ncom{\eeqn}{\end{eqnarray*}}
\ncom{\beq}{\begin{eqnarray}}
\ncom{\eeq}{\end{eqnarray}}
\ncom{\nno}{\nonumber}
\ncom{\rar}{\rightarrow}
\ncom{\Rar}{\Rightarrow}
\ncom{\noin}{\noindent}
\ncom{\bc}{\begin{centre}}
\ncom{\ec}{\end{centre}}
\ncom{\sz}{\scriptsize}
\ncom{\rf}{\ref}
\ncom{\sgm}{\sigma}
\ncom{\Sgm}{\Sigma}
\ncom{\dt}{\delta}
\ncom{\Dt}{Delta}
\ncom{\s}{\underline{s}}
\ncom{\lmd}{\lambda}
\ncom{\Lmd}{\Lambda}
\ncom{\eps}{\epsilon}
\ncom{\pcc}{\stackrel{P}{>}}
\ncom{\dist}{{\rm\,dist}}
\ncom{\sspan}{{\rm\,span}}
\ncom{\re}{{\rm Re\,}}
\ncom{\im}{{\rm Im\,}}
\ncom{\sgn}{{\rm sgn\,}}
\ncom{\ba}{\begin{array}}
\ncom{\ea}{\end{array}}
\ncom{\eop}{\hfill{{\rule{2.5mm}{2.5mm}}}}
\ncom{\eoe}{\hfill{{\rule{1.5mm}{1.5mm}}}}
\ncom{\eof}{\hfill{{\rule{1.5mm}{1.5mm}}}}
\ncom{\hone}{\mbox{\hspace{1em}}}
\ncom{\htwo}{\mbox{\hspace{2em}}}
\ncom{\hthree}{\mbox{\hspace{3em}}}
\ncom{\hfour}{\mbox{\hspace{4em}}}
\ncom{\hsev}{\mbox{\hspace{7em}}}
\ncom{\vone}{\vskip 2ex}
\ncom{\cH}{{\mathcal H}}
\ncom{\vtwo}{\vskip 4ex}
\ncom{\vonee}{\vskip 1.5ex}
\ncom{\vthree}{\vskip 6ex}
\ncom{\vfour}{\vspace*{8ex}}
\ncom{\norm}{\|\;\;\|}
\ncom{\integ}[4]{\int_{#1}^{#2}\,{#3}\,d{#4}}
\ncom{\inp}[2]{\langle{#1},\,{#2} \rangle}
\ncom{\Inp}[2]{\left\langle{#1},\,{#2} \right\rangle}
\ncom{\vspan}[1]{{{\rm\,span}\#1 \}}}
\ncom{\dm}[1]{\displaystyle {#1}}
\ncom{\Hom}{\operatorname{Hom}}
\ncom{\Hol}{\operatorname{Hol}}
\ncom{\Ps}{\mathcal P_{\underline{s}}}
\ncom{\hl}{\mathcal H}

\ncom{\defin} {\overset {\text {\rm def} }{=}}

\newtheorem{theorem}{\bf Theorem}[section]
\newtheorem{defn}{\bf Definition}%[]
\newtheorem{proposition}[theorem]{\bf Proposition}%[section]
\newtheorem{corollary}[theorem]{\bf Corollary}%[section]
\newtheorem{lemma}[theorem]{\bf Lemma}%[section]
\newtheorem{question}[theorem]{\bf Question}

\newtheorem{remark}[theorem]{\bf Remark}%[section]
\def \N{\mathbb{N}}

\def \s{\underline{s}}
\def \m{\underline{m}}
\def \\lambda{\underline{\lambda}}

\renewcommand{\epsilon}{\varepsilon}
\renewcommand{\kappa}{\varkappa}

\begin{document}
\title[Subnormality of $\mathbb K$-Homogeneous operators]{The Subnormality of $\mathbb K$-Homogeneous Multiplication Operators on Bounded Symmetric Domains}
 \author[S. Kumar ]{Surjit Kumar}
 \address[S. Kumar]{Department of Mathematics, Indian Institute of Technology Madras, Chennai 600036, India} 
 \email{surjit@iitm.ac.in}
 \author[M. K. Mal]{Milan Kumar Mal}
 \address[M. K. Mal]{Department of Mathematics, Indian Institute of Technology Madras, Chennai 600036, India}
  \email{ma21d018@smail.iitm.ac.in; milanmal1702@gmail.com }

\thanks{The work of S. Kumar was partially supported by MATRICS grant (Ref No. MTR/2022/000457) and IRG grant (Ref No. ANRF/IRG/2024/000432/MS) of Anusandhan National Research Foundation (ANRF). 
 Support for the work of M. K. Mal was provided in the form of a Prime Minister's Research Fellowship (PMRF / 2502827).}

 \subjclass[2020]{Primary 47A13, 47B20, 47B32, 46E20, Secondary 32M15}
 \keywords{Cartan domain, subnormal tuple, weighted Bergman space, homogeneous operator, Hausdorff moment sequence, the Selberg-Jack symmetric function}

\date{}

\begin{abstract}
   Let $\Omega=G/\mathbb K$ be an irreducible bounded symmetric domain of rank $r$ and dimension $d.$ In this paper, we study the joint subnormality of $d$-tuple of multiplication operators induced by the coordinate functions on reproducing kernel Hilbert spaces of holomorphic functions on $\Omega$ determined by $\mathbb K$-invariant kernels. We introduce the notion of contractive $d$-tuple associated with $\Omega$ and prove that the weighted Bergman shifts on $\Omega$ are contractive $d$-tuple precisely when they are jointly subnormal.  
A characterization of joint subnormality further leads to the study of a class of moment problems, which we refer to as twisted moment problems. We establish that the twisted moment problem is equivalent to an appropriate Hausdorff moment problem. This equivalence provides a moment-theoretic characterization of joint subnormality for the multiplication operators under consideration.
\end{abstract}

\maketitle

\section{Introduction}\label{S1}
The study of subnormal operators reveals a deep and fruitful connection between abstract operator theory and classical function theory, as seen in fundamental examples such as the unilateral shift and the Bergman shift. This interplay extends to the multivariable setting, where the notion of joint subnormality plays a similarly pivotal role. Indeed, the existence of a normal extension endows a commuting operator tuple with a rich functional calculus, thereby enabling function theoretic and spectral properties, for instance, see   \cite{RC1985}, \cite{At1990}, \cite{Cn19912}, \cite{DE2005}. A masterful and detailed exposition of one variable subnormal operators can be found in Conway \cite{Cn1991}.
 
The subnormality of tuples of multiplication operators by the coordinate functions on the weighted Bergman spaces was investigated by Bagchi and Misra \cite{BM1996} for the matrix unit ball, and subsequently extended to irreducible bounded symmetric domains by Arazy and Zhang in \cite{AZ2003}. 
  These operator tuples are homogeneous and have been studied extensively in the past few years (see \cite{MS1990}, \cite{MU2016}, and \cite{KM2019}).
  This motivates us to study the problem of characterizing the subnormality for a broader class of operator tuples that are homogeneous under a maximal compact subgroup $\mathbb K$ of the identity component of the biholomorphic automorphism group of an irreducible bounded symmetric domain. 
A class of $\mathbb{K}$-homogeneous $d$-tuples of operators $\boldsymbol{T}$, for which the joint kernel of $\boldsymbol{T}^*$ is a one-dimensional cyclic subspace for $\boldsymbol{T}$ and the joint point spectrum of $\boldsymbol{T}^*$ contains an irreducible bounded symmetric domain $\Omega \subseteq \mathbb{C}^d$, was studied in \cite{GKP2022} (see also \cite{Up2021}, \cite{HU2023}).
In fact, such operator tuples is unitarily equivalent to the $d$-tuple $\boldsymbol M$ of multiplication operators by the coordinate functions on a reproducing kernel Hilbert space of holomorphic functions on $\Omega$, determined by a $\mathbb K$-invariant reproducing kernel \cite[Theorem 2.3]{GKP2022}.
While several properties of $\boldsymbol{M}$, including boundedness, membership in the Cowen–Douglas class, and its unitary and similarity orbits, were studied in \cite{GKP2022}, the question of its subnormality remains open. In this paper, one of our main results is to establish a criterion for the subnormality of $\boldsymbol M$. In the case of the Euclidean unit ball, a necessary and sufficient condition for subnormality of a certain class of $\mathcal{U}(d)$-homogeneous operator tuples was obtained in \cite[Theorem 5.3]{CY2015}.

For an irreducible bounded symmetric domain $\Omega$ of rank $r$ in $\mathbb C^d$, let $\mathrm{Aut}(\Omega)$ be the group of biholomorphic automorphisms of $\Omega$, equipped with the compact-open topology, that is, the topology of uniform convergence on compact subsets of $\Omega.$
Let $G$ denote the connected component of identity in $\mathrm{Aut}(\Omega).$ It is known that the group $G$ acts transitively on $\Omega.$ Let $\mathbb K$ be the group of linear automorphisms in $G$, then by Cartan's theorem $\mathbb K=\{\phi \in G: \phi(0)=0\}$ is a maximal compact subgroup of $G$ and $\Omega$ is isomorphic to $G/\mathbb K.$
An irreducible bounded symmetric domain can also be realized as an open unit ball of a {\it Cartan factor} $Z \approx \mathbb C^d$. By the classification result of \'E. Cartan \cite{Ca1935}, there are six types of irreducible bounded symmetric domains up to biholomorphic equivalence, of which the first four are called the classical Cartan domains. The other two types of domains are known as {\it exceptional domains}. 
The type of the irreducible bounded symmetric domain of rank $r$ classified by two characteristic multiplicities $a$ and $b.$ Moreover, $d=r+\frac{a}{2}r(r-1)+rb.$ In this paper, $\Omega$ will always denote a classical Cartan domain.

There is a family $\{e_1, \ldots, e_r\}$ of pairwise orthogonal minimal tripotents known as {\it Jordan frame} of $Z$ such that every $z \in Z $ admits a {\it polar decomposition} \beq \label{polardecom} z= k \cdot \sum_{j=1}^r t_j e_j,\qquad k \in \mathbb{K}, \; t_1 \geq \ldots \geq t_r \geq 0.\eeq While $k\in \mathbb K$ in the polar decomposition of $z$ need not be uniquely determined, the ordered singular values $t_1,\ldots,t_r$ are uniquely determined. 
 Furthermore, $t_1 < 1, t_1 =1,$ or $t_1\geq 1$ determines whether $z$ belongs to $\Omega$, $\partial\Omega$ or $\Omega^c$, respectively. Here, $\partial \Omega$ is the topological boundary of $\Omega.$
The domain $\Omega$ is the open unit ball $\{ z \in Z: \|z\| <1\},$ where $\|z\|$ denotes the largest singular value of $z$. Let $e=e_1+\cdots+e_r$ be a maximal tripotent. The set $\{ z \in  Z: z= k \cdot e, k \in \mathbb K\}$ of maximal tripotents forms the {\it Shilov boundary} $S_\Omega$ of $\Omega.$ 
Let $\Delta_r=\{(t_1,\ldots,t_r): 0 \leq t_r \leq \cdots \leq t_1 \leq 1\}$. Note that $\Delta_r$ can be identified as a subset of $\overline{\Omega}$ via the map $\Delta_r \ni (t_1,\ldots, t_r)\mapsto \sum_{i=1}^r t_i e_i\in \overline{\Omega}.$ In view of the polar decomposition \eqref{polardecom}, each $\mathbb K$-orbit in $\overline{\Omega}$ intersects $\Delta_r$ in exactly one point. This yields a canonical identification of the orbit space $\overline{\Omega}/\mathbb K$ with $\Delta_r$ via the quotient map $\pi: \overline{\Omega} \to \overline{\Omega}/\mathbb K\approx \Delta_r$, given by $\pi(z)=$ the $r$-tuple of singular values of $z$ arranged in decreasing order. With this identification, every element in $S_\Omega$ corresponds to the distinguished element $(1,1,\ldots,1)$ in $\Delta_r.$
We refer to \cite{Ls1977}, \cite{AJ1995}, and \cite{UP1996} for more details on bounded symmetric domains. 

Let $\mathcal{B}(\mathcal{H})$ denote the algebra of bounded linear operators on a complex separable Hilbert space $\mathcal{H}.$ By a commuting $d$-tuple $\boldsymbol{T}=(T_1,\ldots, T_d)$, we mean a $d$-tuple of bounded linear operators $T_1,\ldots, T_d\in \mathcal{B}(\mathcal{H})$ such that $T_i T_j=T_jT_i$ for all $1 \leq i,j \leq d.$ For a commuting $d$-tuple $\boldsymbol{T}$, the notations $\sigma(\boldsymbol T)$, $\sigma_p(\boldsymbol T)$ and $\sigma_{\pi}(\boldsymbol T)$ are reserved for the Taylor joint spectrum, the joint point spectrum, and the joint approximate point spectrum of $\boldsymbol T$, respectively. 
%We fix notations $\m,\; \s,\;\underline{\lambda}$ for signatures.

Let $\mathcal{P}(Z)$ denote the space of (analytic) polynomials on the Cartan factor $Z$, equipped with the Fischer-Fock inner product \[\langle p,q\rangle_{\mathcal{F}}:=\frac{1}{\pi^d}\int_{\mathbb{C}^d}p(z)\overline{q(z)}e^{-| z|^2}dm(z),\]
where $dm(z)$ is the Lebesgue measure. The group $\mathbb K$ acts on $\mathcal P(Z)$ via composition $(k\cdot p)(z)=p(k^{-1} \cdot z)$, for all $k\in \mathbb K$ and $p\in \mathcal P(Z).$ 
This action induces a {\it Peter-Weyl decomposition} $\mathcal{P}(Z)=\bigoplus_{\m} \mathcal{P}_{\m }$ into irreducible, mutually $\mathbb{K}$-inequivalent subspaces $\mathcal P_{\m},$ where the indices $\m$ runs over all signatures, $r$-tuple $\m=(m_1,\ldots, m_r)$ of non-negative integers with $m_1\geq m_2\geq \cdots \geq m_r\geq 0$ (see \cite[page 21]{AJ1995}). We reserve the notations $\m$, $\s$, and $\underline{\lambda}$ for signatures, and denote the set of all signatures by $\Vec{\mathbb{N}}^r$. Note that the space $\mathcal{P}_{\m }$ is $\mathbb K$-invariant with respect to the Fischer-Fock inner product. 
Let $d_{\m}$ be the dimension of the space $\mathcal{P}_{\m}.$ Consider an orthonormal basis $\{\psi_{\beta}^{\m}\}_{\beta=1}^{d_{\m}}$ of $\mathcal{P}_{\m}$ with respect to the Fischer-Fock inner product. The Fischer-Fock reproducing kernel $K_{\m}$ of $\mathcal{P}_{\m}$ is given by $$K_{\m}(z,w)=\sum_{\beta} \psi_{\beta}^{\m}(z)\overline{\psi_{\beta}^{\m}(w)}, \qquad z,w\in Z.$$

Note that every element $k\in \mathbb K$ is a $d$-tuple $(k_1,\ldots, k_d)$ of linear polynomials in $d$-variables. Therefore, the group $\mathbb{K}$ acts on a commuting $d$-tuple $\boldsymbol{T}=(T_1, \ldots, T_d)$ of bounded linear operators defined on a complex separable Hilbert space $\mathcal{H}$ via the map $$k\cdot\boldsymbol{T}:=\big(k_1(T_1, \ldots, T_d), \ldots, k_d(T_1, \ldots, T_d)\big). $$
The tuple $\boldsymbol T$ is said to be $\mathbb K$-homogeneous if $\boldsymbol T$ is unitarily equivalent to $k\cdot \boldsymbol T,$ that is, there exists a unitary operator $\Gamma(k)$ on $\mathcal{H}$ such that $T_j \Gamma(k)=\Gamma(k)k_j(\boldsymbol{T}),\;\; j=1,\ldots,d.$ 
In \cite[Theorem 2.3]{GKP2022}, it is shown that a $\mathbb K$-homogeneous operator tuple $\boldsymbol{T}$ satisfying $(i)$ the joint kernel of $\boldsymbol T^*$ is a one dimensional cyclic subspace for $\boldsymbol{T}$, and $(ii)\; \Omega \subseteq \sigma_p(\boldsymbol{T}^*)$ is unitarily equivalent to the tuple $\boldsymbol{M}_z^{(\alpha)}$ of multiplication operators by the coordinate functions on the reproducing kernel Hilbert space $\mathcal{H}^\alpha(\Omega)$ determined by the $\mathbb K$-invariant kernel $$K^\alpha( z, w) = \sum_{\m\in \Vec{\mathbb N}^r} \alpha_{\m} K_{\m} (z, w),\qquad z, w \in \Omega,$$ where $\alpha_{\underline 0}=1, \; \alpha_{\m}>0$ for all $\m.$ The case when $\mathbb K = \mathcal{U}(d),$ this result was established in \cite[Theorem 2.5]{CY2015}. We stress that the notation $K^\alpha(z,w)$ is not intended to indicate a power of ``$K(z,w)$"; instead, it denotes the kernel determined by the sequence $\{\alpha_{\m}\}$.
Note that the inner product on $\mathcal{H}^\alpha(\Omega)$ is related with the Fischer-Fock inner product as follows \beq \label{inner-product relation}\inp{p}{q}_{\mathcal{F}}=\alpha_{\m} \inp{p}{q}_{\alpha}, \qquad p, q\in \mathcal{P}_{\m}. \eeq
 The multiplication operator tuple on the Bergman space $\mathbb A^2(\Omega)$ of $\Omega$ is a classical example of such $\mathbb K$-homogeneous operator tuples. It is known that the Bergman space is a reproducing kernel Hilbert space determined by the (Bergman) kernel $B(z,w)=\Delta (z,w)^{-p},$ where $p=2+a(r-1)+b,$ is the {\it genus} of the domain $\Omega$ (see \cite[Theorem 2.9.8 ]{UP1996}). 
Here, $\Delta(z,w)$ is a $\mathbb K$-invariant sesqui-analytic polynomial, called the {\it Jordan triple determinant}, uniquely determined by the property $\Delta(z,z)=\prod_{i=1}^r (1-t_i^2)$, where $z=k \cdot (\sum_{i=1}^r t_i e_i)$ is the polar decomposition (see \cite[p. 23]{AJ1995}). By Faraut-Kor\'anyi formula \cite[Theorem 3.8]{FK1990}, for any scalar $\nu,$ $$\Delta(z,w)^{-\nu}=\sum_{\m} (\nu)_{\m} K_{\m}(z,w),$$ where $(\nu)_{\m}$ denotes the {\it generalized Pochhammer} symbol $$(\nu)_{\m}=\prod_{j=1}^r \left(\nu-\frac{a}{2}(j-1)\right)_{m_j}=\prod_{j=1}^r \prod_{l=1}^{m_j}\left(\nu-\frac{a}{2}(j-1)+l-1\right).$$ For $\nu \in \{\frac{a}{2}(j-1): j=1,\ldots, r\}\bigcup(\frac{a}{2}(r-1), \infty)$ ({\it Wallach set}), $K^\nu(z,w):=\Delta(z,w)^{-\nu}$ is a positive semi-definite kernel on $\Omega$ (see \cite[Corollary 4.4]{AJ1995}).
Consequently, the kernel $K^{\nu}(z,w)$ induces the reproducing kernel Hilbert space $\mathcal{H}^\nu(\Omega)$, known as the {\it weighted Bergman space}. In particular, the choices $\nu=\frac{d}{r}$ and $\nu=\frac{a}{2}(r-1)+\frac{d}{r}+1$ correspond to the Hardy space on the Shilov boundary and the standard Bergman space, respectively. The $d$-tuple $\boldsymbol{M}_z^{(\nu)}$ of multiplication operators by the coordinate functions on $\mathcal{H}^\nu(\Omega)$ is bounded whenever $\nu > \frac{a}{2}(r-1)$. Moreover, these tuples are $\mathbb{K}$-homogeneous, and in fact, they are homogeneous (under $G$) $d$-tuple of operators (see \cite{BM1996}, \cite{AZ2003}).

We now briefly outline the organization of the paper.
In Section \ref{S2}, we introduce the notion of a contraction, referred to as {\it contractive $d$-tuple}, associated with $\Omega$ for a commuting $d$-tuple $\boldsymbol{T}$ which extends the notion of spherical (column) contraction. We prove that the Taylor joint spectrum of a contractive $d$-tuple is contained in $\overline{\Omega}.$ Furthermore, we identify the weighted Bergman shifts on $\Omega$ that are contractive $d$-tuples in Theorem \ref{contractionpoint}.

Recall that a commuting $d$-tuple $\boldsymbol S=(S_1, \cdots, S_d)$ defined on a Hilbert space $\mathcal{H}$ is said to be (jointly) {\it subnormal} if there exist a Hilbert space ${\mathcal K}$ containing ${\mathcal H}$ and a commuting $d$-tuple $\boldsymbol N=(N_1, \cdots, N_d)$ of normal operators on $\mathcal{K}$ such that $\mathcal{H}$ is invariant under each $N_i$ and $S_i=N_i|_{\mathcal{H}}$ for all $1\leq i \leq d.$ 

In Section \ref{S3}, we obtain a criterion for subnormality of a contractive $d$-tuples $\boldsymbol M_z^{(\alpha)}$ in Theorem \ref{Main theorem 1}. One of these characterizations, part $(iii)$ of Theorem \ref{Main theorem 1}, leads us to consider the following type of moment problem: for a given sequence $\{b_{\m}\}_{\m \in \Vec{\mathbb N}^r },$ when does there exists a measure $\nu$ supported on $\Delta_r$ such that $\int_{\Delta_r} S_{\m}^{a/2}(x) d\nu(x)=b_{\m}\; ?$ Here, $S_{\m}^{a/2}(x)$ denotes the Selberg-Jack symmetric function. The resolution of this problem is equivalent to solving a certain Hausdorff moment problem, as shown in Theorem \ref{moment problem condition}. Ultimately, this equivalence yields a complete characterization for the subnormality of contractive $d$-tuple $\boldsymbol{M}_z^{(\alpha)}$ in Corollary \ref{complete characterization}.

Section \ref{S4} is devoted to discussing a few applications of Theorem \ref{Main theorem 1}. Note that the tuple $\boldsymbol{M}_z^{(\nu)}$ is subnormal if and only if $\nu=\frac{d}{r}, \frac{d}{r}+\frac{a}{2},\ldots, \frac{d}{r}+\frac{a}{2}(r-1)$ or $\nu > \frac{d}{r}+\frac{a}{2}(r-1) $~\cite[Theorem 6.1]{AZ2003}. We provide an independent proof of this result (see Theorem \ref{subnormality alternative}) using the equivalence of parts $(i)$ and $(iii)$ of Theorem \ref{Main theorem 1}. More precisely, we construct the measure $\mu$ supported on $\overline{\Omega}$, as described in \eqref{measure for cont} and \eqref{measure for discrt}, corresponding to the continuous range and discrete values of $\nu$, respectively. For the continuous range of $\nu$, we use the generalized Selberg integral formula \cite[Theorem 1]{KD1997} to verify part $(iii)$ of Theorem \ref{Main theorem 1}.
However, the formula \cite[Theorem 1]{KD1997} is not applicable with the measures given by \eqref{measure for discrt} for the discrete values of $\nu$. To overcome this, we use a key identity obtained in Proposition \ref{important lemma} to prove Theorem \ref{subnormality alternative}. In an another application, we discuss the subnormality of a class of operator tuples $\boldsymbol{M}_z^{(\alpha)}$ defined on $\mathcal{H}^\alpha(\Omega)$ with $\alpha_{\m}=\frac{(\nu)_{\m}}{(\eta)_{\m}}$ for any real numbers $\nu, \eta >1$, where $\Omega$ is a rank $2$ type-I domain.

\section{Cartan contraction}\label{S2}
In this section, we introduce and study a notion of contraction associated with Cartan domains. This concept is motivated by, and developed in analogy with, the notion of spherical contraction that arises in the context of the Euclidean open unit ball.
%%We also investigate the corresponding row contractive analogue.

The Jordan triple determinant $\Delta(z,w)$ associated with the classical Cartan domain $\Omega$ is a $\mathbb K$-invariant sesqui-analytic polynomial uniquely determined by the property $\Delta(z,z)=\prod_{i=1}^r (1-t_i^2)$, where $z=k \cdot (\sum_{i=1}^r t_i e_i)$ is the polar decomposition (see \cite[p. 23]{AJ1995}).
By Faraut-Kor\'anyi formula \cite[Theorem 3.8]{FK1990}, $\Delta(z,w)$ admits the following decomposition
\[\Delta(z,w)=\sum_{\ell =0}^r (-1)^{\ell} \Delta^{(\ell)}(z,w),\qquad z,w \in \Omega,\] where $(\ell):= (1,\ldots,1,0,\ldots,0)$ is the signature with first $\ell$ many ones. The sesqui-analytic polynomials $\Delta^{(\ell)}(z,w)$ are homogeneous of bi-degree $(\ell,\ell)$ given by
\beq \label{delk} \Delta^{(\ell)} (z,w)= (-1)^{\ell} (-1)_{(\ell)} K_{(\ell)}(z,w)= \prod_{j=1}^{\ell} (1+\frac{a}{2}(j-1))\sum_\alpha \psi_\alpha^{(\ell)}(z) \overline{\psi_\alpha^{(\ell)}(w)}.\eeq
Note that $\Delta^{(1)}(z,w)=\sum_{i=1}^dz_i \overline{w}_i.$ The case when rank $r=1$, $\Delta(z,w)=1-\sum_{i=1}^dz_i \overline{w}_i.$

By combining \eqref{delk} with \cite[Lemma 3.2]{FK1990}, we get
\beqn \Delta(z,z)=\sum_{\ell=0}^r(-1)^{\ell} \Delta^{(\ell)}(\sum_{j=1}^r t_j^2 e_j, e)=\prod_{j=1}^r(1-t_j^2) =\sum_{\ell=0}^r (-1)^{\ell}\sum_{1 \leq j_1 <\ldots < j_{\ell}\leq r} t_{j_1}^2 \cdots t_{j_{\ell}}^2.\eeqn 
Consequently, $\Delta^{(\ell)}(z,z)=\sigma_\ell(t^2),\; \ell=1,\ldots,r,$ where $t^2=(t_1^2,\ldots,t_r^2).$ Here, $\sigma_{\ell}(t)=\sum_{1 \leq j_1 < \cdots < j_\ell \leq r} t_{j_1}\cdots t_{j_\ell}$ denote the $\ell$-th elementary symmetric polynomial in $t=(t_1,\ldots,t_r).$
Set $\sigma_0(t)=1,$ then $\prod_{i=1}^r(1+t_i x)=\sum_{\ell=0}^r \sigma_\ell(t)x^\ell.$ Let $u_i=1-t_i$ for $i=1,\ldots,r$ and $u=(u_1,\ldots,u_r).$ We get \beqn \sum_{n=0}^r \sigma_n(u) x^n &=& \prod_{i=1}^r (1+u_i x)= \prod_{i=1}^r (1+(1-t_i)x)=\prod_{i=1}^r (1+x-t_i x) \\
&=& \prod_{i=1}^r  \left(1+x\right) \left(1+t_i \left(\frac{-x}{1+x}\right)\right)= \sum_{\ell=0}^r (-1)^\ell \sigma_\ell(t) x^\ell (1-x)^{r-\ell}\\
&=& \sum_{n=0}^r \left( \sum_{\ell=0}^n (-1)^\ell \binom{r-\ell}{n-\ell}\sigma_\ell(t)\right) x^n.\eeqn
 Therefore, \beq \label{sym u-t}\sigma_n(u)= \sum_{\ell=0}^n (-1)^\ell \binom{r-\ell}{n-\ell}\sigma_\ell(t),\;\; n=1,\ldots,r.\eeq

The following result describes the classical Cartan domain $\Omega$ in terms of $\Delta^{(\ell)}(z,z).$ This description is known and recorded in \cite[Proposition 4.16]{Ls1977}. For completeness, we provide a proof that is consistent with the terminology and setup adopted in this paper.

\begin{lemma}\label{domain description}
Let $\Omega$ be a classical Cartan domain. Then
\begin{align*}
   & \Omega =\left\{ z\in \mathbb C^d: \sum_{\ell=0}^n (-1)^\ell \binom{r-\ell}{n-\ell} \Delta^{(\ell)}(z,z) > 0\;\; \text{for all } n=1,\ldots, r\right\} \; \text{ and}\\
&\overline{\Omega} =\left\{ z\in \mathbb C^d: \sum_{\ell=0}^n (-1)^\ell \binom{r-\ell}{n-\ell} \Delta^{(\ell)}(z,z) \geq 0\;\; \text{for all } n=1,\ldots, r\right\}.
\end{align*}
\end{lemma}
\begin{proof}
Let $z=k \cdot \left(\sum_{i=1}^r t_i e_i\right)$ be the polar decomposition of $z\in Z.$ 
Consider $u_i=1-t_i^2$ for $i=1,\ldots,r.$ Then, for each $1 \leq n \leq r,$ \beqn f_n(z)&:=&\sum_{\ell=0}^n (-1)^\ell \binom{r-\ell}{n-\ell} \Delta^{(\ell)}(z,z)
= \sum_{\ell=0}^n (-1)^\ell \binom{r-\ell}{n-\ell} \sigma_\ell(t^2)\overset{\eqref{sym u-t}}{=} \sigma_n(u).\eeqn
 Consider the polynomial $$p(x):=\prod_{i=1}^r (u_i+x)=x^r \prod_{i=1}^r \left ( 1+ u_i \frac{1}{x}\right)=\sum_{n=0}^r \sigma_n(u) x^{r-n}.$$ If $f_n(z) =\sigma_n(u)\geq 0$ for all $n=1,\ldots,r$, then roots of $p(x)$ must be non-positive. On the other hand, for each $i=1,\ldots, r$, $x=-u_i$ are roots of $p(x)$. Therefore, $u_i\geq 0,$ that is, $ t_i \leq 1$ for all $i=1,\ldots, r.$ This shows that $z\in \overline{\Omega}$. The reverse inclusion is immediate from the fact that $u_i \geq 0$ for all $i=1,\ldots,r$ implies that $f_n(z)=\sigma_n(u)\geq 0$ for all $n=1,\dots,r.$

 Let $f_n(z)>0$ for all $n=1,\ldots,r.$ The first part of the proof shows that $z\in \overline{\Omega}.$ Since $f_r(z)=\prod_{i=1}^r(1-t_i) >0,$ none of the $t_i$ is equal to $1,$ that is, $z\in \Omega.$ 
 For the reverse inclusion, let $z\in \Omega$, then $t_i <1$ for each $i.$ Thus $u_i=1-t_i^2 >0$ for all $i$, and therefore $f_n (z)=\sigma_n(u) >0$ for all $n=1,\ldots,r.$
  This completes the proof.
\end{proof}

For any sesqui-analytic polynomial $p(z,w)=\sum_{\alpha, \beta\in \mathbb Z_+^d} a_{\alpha\beta} z^\alpha \overline{w}^\beta$ and $d$-tuple $\boldsymbol T$ of commuting bounded linear operator $T_1, \ldots, T_d,$ we set 
$$p(z,w)(\boldsymbol T, \boldsymbol T^*):=\sum_{\alpha, \beta \in \mathbb Z_+^d} a_{\alpha \beta} \boldsymbol T^{*\beta} \boldsymbol T^\alpha.$$
Therefore, $\sum_{i=1}^n p_i(z)\overline{q_i(w)}(\boldsymbol T,\boldsymbol T^*)=\sum_{i=1}^n q_i(\boldsymbol T)^* p_i(\boldsymbol T)$ for polynomials $p_1,\ldots, p_n,q_1,\ldots, q_n$ in $d$-variables. 
 We are now ready to define a notion of a contraction associated with the classical Cartan domain.
\begin{defn}
 Let $\boldsymbol T=(T_1,\ldots, T_d)$ be a commuting $d$-tuple of bounded linear operators defined on a Hilbert space $\mathcal{H}.$ Then $\boldsymbol T$ is said to be a Cartan contraction associated with the domain $\Omega$ of rank $r$ if it satisfies the following relations $$\sum_{\ell=0}^k (-1)^\ell \binom{r-\ell}{k-\ell}\Delta^{(\ell)} (z,w)(\boldsymbol T,\boldsymbol T^*) \geq 0 \; \; \text{ for all } k=1,\ldots,r.$$
\end{defn}

 For the remainder of this paper, a commuting $d$-tuple that is a Cartan contraction associated with the domain $\Omega$ will be referred to as a {\it contractive $d$-tuple}. When $r=1$, the above notion reduces to $(1-\sum_{i=1}^dz_i \bar{w}_i)(\boldsymbol T, \boldsymbol T^*) \geq 0,$ that is, $\boldsymbol T$ is a spherical contraction.

  \begin{remark}
      Note that $$\sum_{\ell=0}^k (-1)^\ell \binom{r-\ell}{k-\ell}\Delta^{(\ell)} (z,w)(\boldsymbol T,\boldsymbol T^*) = 0 \; \; \text{ for all } k=1,\ldots,r$$ if and only if $\Delta^{(\ell)}(z,w)(\boldsymbol T, \boldsymbol T^*)=\binom{r}{\ell} I$ for all $\ell=1,\ldots,r,$ equivalently, $\boldsymbol T$ is a Cartan isometry associated to the domain $\Omega$ (\cite[Theorem 2.4]{KMP2025}).
  \end{remark} 

  It is known that the Taylor joint spectrum of a spherical contraction is contained in the closed Euclidean unit ball. In a similar spirit, we establish an analogous result in the setting of classical Cartan domains.
\begin{proposition}\label{spectrumcontraction}
    If $\boldsymbol T$ is a contractive $d$-tuple, then the Taylor joint spectrum $\sigma(\boldsymbol T )$  of $\boldsymbol T$ is  contained in $ \overline{\Omega}.$
\end{proposition}
\begin{proof}
     Let $\sigma_\pi(\boldsymbol T)$ be the approximate point spectrum of $\boldsymbol T$. If $z \in \sigma_{\pi}(\boldsymbol{T})$ then there exists a sequence of unit vectors $x_n$ such that $\lim_{n \to \infty}(\boldsymbol T^\alpha- z^\alpha)x_n= 0$ for all $\alpha \in \N^d$ (see \cite{RC1988}).
 So, for any polynomial  $p$  in $d$-variables,  $\lim_{n\to \infty}\big(p(\boldsymbol{T})-p(z)\big)x_n=0$ (see \cite[p. 187]{GR2006}). Since $\big |\|p(\boldsymbol{T})x_n\|-|p(z)|\big| \leq \|p(\boldsymbol{T})x_n -p(z)x_n\|,$ it follows that $\|p(\boldsymbol{T})x_n\| \to |p(z)|$ as $n \to \infty$.
 Thus, if $z \in \sigma_{\pi}(\boldsymbol T)$, then for each $1 \leq  k \leq r$,
 $$\sum_{\ell=0}^k (-1)^\ell \binom{r-\ell}{k-\ell} \Delta^{(\ell)}(z,z)= \lim_{n \to \infty} \sum_{\ell=0}^k (-1)^\ell \binom{r-\ell}{k-\ell} \left \langle \Delta^{(\ell)} (z,w)(\boldsymbol T, \boldsymbol T^*)x_n, x_n \right\rangle\geq 0.$$
 By Lemma \ref{domain description}, $\sigma_{\pi}(\boldsymbol T)\subseteq \overline{\Omega}.$ Note that $\sigma(\boldsymbol T) $ is contained in the polynomial convex hull of $\sigma_{\pi}(\boldsymbol T)$ (see \cite{SZ1974}).
 Since $\overline{\Omega}$ is polynomially convex (\cite[p. 53]{UP1996}), $\sigma(\boldsymbol T) \subseteq \overline{\Omega}.$
\end{proof}

%Now, we see some natural examples of Cartan contractions. 
We set the following notations for our convenience. For any subset $L$ of $R=\{1,2\ldots,r\},$ let $\chi_L$ be the $r$-tuple with only nonzero entry $1$ at the $i$th place if $i \in L.$ Let $(x)^{\underline \ell}=x(x-1)\ldots (x-\ell+1)$  and $(x)^{\bar{\ell}}=x(x+1)\cdots (x+\ell-1)$ denote the falling and rising factorials, respectively.

Recall that the weighted Bergman space $\mathcal{H}^\nu(\Omega)$ is a reproducing kernel Hilbert space with the reproducing kernel $K^\nu(z,w)=\Delta(z,w)^{-\nu}$ and $\boldsymbol M^{(\nu)}_{z}$ denotes the tuple of multiplication operators by the coordinate functions on $\mathcal{H}^{\nu}(\Omega),$ where $\nu \in \{\frac{a}{2}(j-1): j=1,\ldots,r\}\cup(\frac{a}{2}(r-1), \infty)$.
For each $1\leq \ell \leq r,$ the operator $\Delta^{(\ell)}(z,w)(\boldsymbol M_z^{(\nu)}, \boldsymbol {M_z^{(\nu)*}})$ is block diagonal with respect to the decomposition $\oplus_{\m} \mathcal{P}_{\m}$ (\cite[Theorem 22]{Up2021}). In fact, for any $p \in \mathcal{P}_{\m},$   
$$ \Delta^{(\ell)}(z,w)(\boldsymbol M_z^{(\nu)}, \boldsymbol {M_z^{(\nu)*}})p=\sum_{|L|=\ell} c_{\m}(L) \frac{(\frac{d}{r})_{\m+\chi_L}}{(\frac{d}{r})_{\m}}\frac{(\nu)_{\m}}{(\nu)_{\m+\chi_L}}\;p,$$
where $c_{\m}(L)=\prod_{i \in L \not \ni j} \frac{m_i-m_j+\frac{a}{2}(j-i+1)}{m_i-m_j+\frac{a}{2}(j-i)}.$ Note that $c_{\m}(L)=0$, whenever $\m+\chi_L$ is not a signature. In the following result, we determine those values of $\nu$ for which $\boldsymbol M_z^{(\nu)}$ is a contractive $d$-tuple.

\begin{theorem}\label{contractionpoint}
      $\boldsymbol M_z^{\nu}$ is a contractive $d$-tuple if and only if 
      $$\nu \in \left\{\frac{d}{r}, \frac{d}{r}+\frac{a}{2}, \cdots, \frac{d}{r}+\frac{a}{2}(r-1)\right\}\bigcup \left( \frac{d}{r}+\frac{a}{2}(r-1), \infty\right).$$
 \end{theorem}
\begin{proof}
    Recall that $\boldsymbol M_z^{(\nu)}$ is bounded if and only if $\nu > \frac{a}{2}(r-1).$ Since a contractive $d$-tuple is necessarily bounded, we consider $\nu >\frac{a}{2}(r-1).$ 
     Let $\boldsymbol M_z^{(\nu)}$ be a contractive $d$-tuple.
     %Fix $\m=(0,0,\ldots,0).$ 
     If $L=\{1,2,\ldots,\ell\},$ then for $\m=(0,0,\ldots,0)$
     \beqn c_{\underline 0}(L) &=& \prod_{i=1}^\ell \prod_{j=\ell+1}^r \frac{j-i+1}{j-i}= \prod_{i=1}^\ell \frac{r-i+1}{\ell+1-i}=\binom{r}{\ell}.\eeqn
     On the other hand, if $L\subseteq \{1,2,\ldots, r\}$ with $|L|=\ell$ and  $L\neq \{1,2,\ldots,\ell\},$ then there exists $ i\in L$ such that $i\geq 2$ and $i-1\notin L$. Taking $j=i-1,$ we get $\frac{j-i+1}{j-i}=0$ and hence $c_{\underline 0}(L)=0.$ Thus, $$c_{\underline 0}(L)=\begin{cases}
         \binom{r}{\ell}, &\text{ if } L=\{1,2,\ldots,\ell\},\\
         0, &\text{ otherwise}.
     \end{cases}$$
 Thus, for any constant polynomial $p$ $$\Delta^{(\ell)}(\boldsymbol M_z^{(\nu)}, \boldsymbol M_z^{(\nu) *})p=\binom{r}{\ell} \prod_{j=1}^\ell \frac{\frac{d}{r}-\frac{a}{2}(j-1)}{\nu-\frac{a}{2}(j-1)}\;p =\binom{r}{\ell} \prod_{j=1}^\ell\frac{\frac{d/r}{a/2}-j+1}{\frac{\nu}{a/2}-j+1}\;p=\binom{r}{\ell}\frac{\left(\frac{d/r}{a/2}\right)^{\underline \ell}}{\left(\frac{\nu}{a/2}\right)^{\underline \ell}} \;p.$$
     By using the identities:  $\binom{r-\ell}{k-\ell} \binom{r}{\ell}=\binom{r}{k} \binom{k}{\ell}, \;\; (x)^{\underline \ell}=(-1)^\ell (-x)^{\bar{\ell}},\;\; (-1)^\ell\binom{k}{\ell}=\frac{(-k)^{\bar{\ell}}}{\ell!}$, we get 

     \beqn \sum_{\ell=0}^k (-1)^\ell \binom{r-\ell}{k-\ell} \Delta^{(\ell)}(\boldsymbol M_z^{(\nu)}, \boldsymbol M_z^{(\nu)*}) \; p&=&\sum_{\ell=0}^k (-1)^\ell \binom{r-\ell}{k-\ell} \binom{r}{\ell}\frac{\left(\frac{d/r}{a/2}\right)^{\underline \ell}}{\left(\frac{\nu}{a/2}\right)^{\underline \ell}} \; p\\
     &=& \binom{r}{k}\sum_{\ell=0}^k \frac{(-k)^{\bar{\ell}}}{\ell!} \frac{\left(-\frac{d/r}{a/2}\right)^{\bar{\ell}}}{\left(-\frac{\nu}{a/2}\right)^{\bar{\ell}}}\; p.\eeqn

    Let $_2F_1(a,b; c;z)=\sum_{\ell=0}^\infty \frac{(a)^{\bar{\ell}}(b)^{\bar{\ell}}}{(c)^{\bar{\ell}} }\frac{z^\ell}{\ell!}$ be the hypergeometric series. By using the standard Chu-Vandermonde identity: $_2F_1(-k, b;c;1)=\frac{(c-b)^{\bar{k}}}{(c)^{\bar{k}}},$ we get \beqn \sum_{\ell=0}^k \frac{(-k)^{\bar{\ell}}}{\ell!} \frac{\left(-\frac{d/r}{a/2}\right)^{\bar{\ell}}}{\left(-\frac{\nu}{a/2}\right)^{\bar{\ell}}}&=&_2F_1\left(-k, -\frac{d/r}{a/2}; -\frac{\nu}{a/2};1\right)=\frac{\left(\frac{d/r}{a/2}-\frac{\nu}{a/2}\right)^{\bar{k}}}{\left(-\frac{\nu}{a/2}\right)^{\bar{k}}}\\
     &=& \frac{\left(\frac{\nu-\frac{d}{r}}{a/2}\right)^{\underline{k}}}{\left(\frac{\nu}{a/2}\right)^{\underline{k}}}\\
     &=& \frac{\left(\nu-\frac{d}{r}\right)\left(\nu-\frac{d}{r}-\frac{a}{2}\right)\cdots \left(\nu-\frac{d}{r}-\frac{a}{2}(k-1)\right)}{\nu \left(\nu-\frac{a}{2}\right)\cdots\left(\nu-\frac{a}{2}(k-1)\right)}.\eeqn
The denominator in the last expression is strictly positive as $\nu> \frac{a}{2}(r-1).$ Since $\boldsymbol M_z^{(\nu)}$ is a contractive $d$-tuple, $\sum_{\ell=0}^k (-1)^\ell \binom{r-\ell}{k-\ell} \Delta^{(\ell)}(\boldsymbol M_z^{(\nu)}, \boldsymbol M_z^{(\nu)*}) \geq 0$ for all $1 \leq k \leq r$ implies that the parameter $\nu$ either belongs to the discrete set $\{\frac{d}{r}, \frac{d}{r}+\frac{a}{2}, \cdots, \frac{d}{r}+\frac{a}{2}(r-1)\}$ or satisfies $\nu > \frac{d}{r}+\frac{a}{2}(r-1).$

Conversely, if $\nu$ is either any of the discrete values $\frac{d}{r},\frac{d}{r}+\frac{a}{2}, \ldots, \frac{d}{r}+\frac{a}{2}(r-1)$ or $\nu > \frac{d}{r}+\frac{a}{2}(r-1)$, then by \cite[Theorem 6.1]{AZ2003}, $\boldsymbol M_z^{(\nu)}$ is subnormal. Consequently, for these $\nu$, an application of the spectral theorem for the minimal normal extension of $\boldsymbol M_z^{(\nu)}$ together with Lemma \ref{domain description} shows that $\boldsymbol M_z^{(\nu)}$ is a contractive $d$-tuple. This completes the proof of the theorem.
\end{proof}

Recall that $\boldsymbol M^{(\alpha)}_z=(M^{(\alpha)}_{z_1},\ldots, M_{z_d}^{(\alpha)})$ is the $d$-tuple of multiplication operators by the coordinate functions on $\mathcal{H}^\alpha(\Omega),$ where $H^{\alpha}(\Omega)$ is the reproducing kernel Hilbert space associated with the $\mathbb K$-invariant kernel $$K^{\alpha}(z,w)=\sum_{\s} \alpha_{\s} K_{\s}(z,w), \qquad z,w \in \Omega, \;\alpha_{\underline 0}=1, \;\alpha_{\s} >0.$$
It is well known that the operators $\Delta^{(\ell)} (z,w)(\boldsymbol{M}_z^{(\alpha)}, \boldsymbol{M}_z^{(\alpha)*})$ acting on $\mathcal{H}^\alpha(\Omega)$ are block diagonals for all $1 \leq \ell \leq r$ with respect to the decomposition $\oplus_{\m\in \vec{\mathbb N}^r} \mathcal P_{\m}$
(cf. \cite[p. 5]{Up2021}, \cite[Lemma 3.1]{GKP2022}). 
Following the approach of \cite[Theorem 22]{Up2021}, we obtain the so called second eigenvalue formula for $\boldsymbol{M}_z^{(\alpha)}$.

\begin{proposition}
	Let $\Vec{\N^r}$ be the set of all signatures. Then for any $1 \leq \ell \leq r$, 
$$ \Delta^{(\ell)}(z,w)(\boldsymbol{M}_z^{(\alpha)},\boldsymbol{M}_z^{(\alpha)*})p=\sum_{\substack{|L|=\ell \\  \m+\chi_L \in \Vec{\N^r}}}\frac{(d/r)_{\m+\chi_L}}{(d/r)_{\m}}\frac{\alpha_{\m}}{\alpha_{\m+\chi_L}} c_{\m}(L)\;p, \qquad p \in \mathcal{P}_{\m} .$$  
\end{proposition}

\begin{corollary}
    The multiplication operator tuple $\boldsymbol M_z^{(\alpha)}$ defined on $\mathcal{H}^\alpha(\Omega)$ is a contractive $d$-tuple if and only if for any signature $\m,$ 
    \beq \label{eq 12}\sum_{\ell=0}^k (-1)^{k}\binom{r-\ell}{k-\ell}\sum_{|L|=\ell}\frac{(d/r)_{\m+\chi_L}}{(d/r)_{\m}}\frac{\alpha_{\m}}{\alpha_{\m+\chi_L}} c_{\m}(L) \geq 0,\; 1\leq k\leq r.\eeq
    Furthermore, $\boldsymbol M_z^{(\alpha)}$ is a subnormal contractive $d$-tuple if and only if $\boldsymbol M_z^{(\alpha)}$ is subnormal and $\sigma(\boldsymbol M_z^{(\alpha)})=\overline{\Omega}.$
\end{corollary}

\section{Subnormality of $\mathbb K$-homogeneous operator tuples}\label{S3}
Recall that $\mathcal{H}^\alpha(\Omega)$ is the reproducing kernel Hilbert space of holomorphic functions on $\Omega$ determined by the $\mathbb K$-invariant kernel $K^\alpha(z,w)=\sum_{\s} \alpha_{\s}K_{\s}(z,w),$ where $\{\alpha_{\s}\}$ is a sequence of positive real numbers with $\alpha_{\underline 0}=1.$ One of the main results of this section establishes a criterion for the subnormality of the multiplication operator tuple $\boldsymbol M_z^{(\alpha)}$ on $\mathcal{H}^\alpha(\Omega).$ This characterization gives rise to a type of moment problem for a sequence of positive real numbers indexed by signatures. We refer to it as the {\it twisted moment problem}.

\subsection{Criterion for Subnormality}
We first briefly discuss subnormality in the rank-one case, namely, the Euclidean open unit ball $\mathbb B$. In this case, a subnormality criterion is obtained in \cite[Theorem 5.3]{CY2015}. However, we follow the approach of \cite{BM1996}, which characterizes the subnormality of weighted Bergman shifts on type-I domains.

Consider an $\mathcal U(d)$-invariant kernel $K^\delta(z,w)=\sum_{k\geq 0}\delta_k \inp{z}{w}^k$ with $\delta_0=1$ and $\delta_k>0$ for all $k.$ Let $\mathscr{M}_z$ be the $d$-tuple of multiplication operators by the coordinate functions on the reproducing kernel Hilbert space $\mathcal{H}^\delta(\mathbb B).$
If $\mathscr{M}_z$ is a subnormal spherical contraction, then there exists a probability measure $\mu$ supported on $\overline{\mathbb B}$ such that $\inp{f}{g}_\delta=\int_{\overline{\mathbb B}}f \bar{g}d\mu$ for all polynomials $f$ and $ g.$  Since the inner product on $\mathcal{ H}^\delta(\mathbb B)$ is $\mathcal{U}(d)$-invariant, by the Stone-Weierstrass theorem, the measure $\mu$ is $\mathcal{U}(d)$-invariant. Note that for any $\beta \in \mathbb Z_+^d$, $\|z^\beta\|_\delta^2= \frac{\beta!}{\delta_{|\beta|}\cdot |\beta|!}=\int_{\overline{\mathbb B}} |z^\beta|^2 d\mu.$ In particular, \beqn \|z_1^k\|_\delta^2=\int_{\overline{\mathbb B}} |z_1^k|^2 d\mu(z)=\frac{1}{\delta_k},\qquad \text{ for any } k\in \mathbb Z_+.\eeqn  
Conversely, assume that there is an $\mathcal{U}(d)$-invariant probability measure $\mu$ supported on $\overline{\mathbb B}$ with $\frac{1}{\delta_k}=\int_{\overline{\mathbb B}} |z_1^k|^2 d\mu(z)$ for all $k \in \mathbb Z_+.$ Since the space of homogeneous polynomials $\mathcal{P}_k$ of degree $k$ in $d$-variables is $\mathrm{span}\{ (u\cdot z)_1^k: u\in \mathcal{U}(d)\},$ the measure $\mu$ give rise to an $\mathcal{U}(d)$-invariant inner product on $\mathcal{P}_k.$ The Schur Lemma along with the condition $\frac{1}{\delta_k}=\int_{\overline{\mathbb B}} |z_1^k|^2 d\mu(z)$, gives that the inner product arises from the measure $\mu$ coincides with that of $\mathcal{H}^\delta(\mathbb B).$ Therefore, the $d$-tuple $\mathscr{M}_z$ is a subnormal spherical contraction.

Note that the set of $\mathcal{U}(d)$-invariant probability measures $\mu$ supported on $\overline{\mathbb B}$ is in one-to-one correspondence with the set of all probability measure $\nu$ supported on $[0,1]$ given by $\nu=\mu\circ \pi^{-1},$ where $\pi:\overline{\mathbb B} \to [0,1]$ is the map $\pi(z)=\sqrt{\sum_{i=1}^ d |z_i|^2}.$ 
Therefore, for any $\mathcal{U}(d)$-invariant probability measure $\mu$ supported on $\overline{\mathbb B}$ and $f\in L^1(\mu),$  we have $$\int_{\overline{\mathbb B}} fd\mu=\int_{[0,1]} \left( \int _{\mathcal{U}(d)} f(u\cdot (t,0,\ldots,0))du \right)d\nu(t),$$
where $du$ is the unique Haar measure on $\mathcal{U}(d).$ 
Let $\sigma$ be the normalized surface area measure on the unit sphere $\partial \mathbb B.$ Then, for any polynomial $p$ and $q$, $\int_{\partial \mathbb B} p \overline{q} d\sigma=\int_{\mathcal{U}(d)} p (u\cdot(1,0,\ldots,0))\overline{q (u\cdot(1,0,\ldots,0))}du$ (see \cite[p. 2.8]{BM1996}). 
Let $\mu$ be an $\mathcal{U}(d)$-invariant measure with $\frac{1}{\delta_k}=\int_{\overline{\mathbb B}}|z_1^k|^2 d\mu(z)$ for all $k \in \mathbb Z_+.$
Consider $\phi_k(z)=z_1^k,$ then
\beqn \int_{\overline{\mathbb B}}|\phi_k|^2 d\mu &= &\int_{[0,1]} \left(\int_{\mathcal{U}(d)} |\phi_k(u\cdot (t,0,\ldots,0))|^2 du\right) d\nu(t)\\
&=&\int_{[0,1]}t^{2k} \left( \int_{\partial \mathbb B} |\phi_k|^2  d\sigma \right) d\nu(t)=\frac{k!}{(d)_k}\int_{[0,1]} t^{2k} d\nu(t).\eeqn
Therefore, $\frac{1}{\delta_k}=\int_{[0,1]}\frac{k!}{(d)_k} t^{2k} d\nu(t)$, 
that is, the sequence $\left\{\frac{(d)_k}{\delta_k  k!}\right\}$ is a Hausdorff moment sequence. On the other hand, if $\left\{\frac{(d)_k}{\delta_k  k!}\right\}$ is a Hausdorff moment sequence with the representing measure $\nu,$ then we get a $\mathcal{U}(d)$-invariant probability measure $\mu$ as follows: $$\mu(A)=\int_{[0,1]} \left(\int_{\mathcal{U}(d)} \chi_{u^{-1}\cdot A}(t,0,\ldots,0)du\right) d\nu(t),$$ for all Borel subset $A$ of $\overline{\mathbb B}.$ Thus, $$\int_{\overline{\mathbb B}} |z_1^k|^2 d\mu=\int_{[0,1]}\left(\int_{\mathcal{U}(d)} |\phi_k(u \cdot (t,0,\ldots,0))|^2 du\right) d\nu(t)=\frac{1}{\delta_k},\qquad k\in \mathbb Z_+.$$

The preceding discussion establishes the equivalence of the following statements:
\begin{itemize}
        \item [(i)] The multiplication operator tuple $\mathscr{M}_z$ on $\mathcal{H}^\delta(\mathbb B)$ is subnormal.
        \item [(ii)] There is a unique $\mathcal{U}(d)$-invariant measure $\mu$ supported on $\overline{\mathbb B}$ such that $$\int_{\overline{\mathbb B}}|z_1^k|^2 d\mu(z)=\frac{1}{\delta_k},\qquad \text{ for all } k \in \mathbb Z_+.$$
        \item [(iii)] The sequence $\left\{\frac{(d)_k}{\delta_k  k!}\right\}$ is a Hausdorff moment sequence.
\end{itemize}

To derive a subnormality criterion for $\mathbb K$-invariant kernels,
we adopt the same strategy used in the above $\mathcal U(d)$-invariant case. To this end, we first recall the notion of elementary spherical functions.
 For any signature $\m$, there is, up to a scalar multiple, a unique function $\phi_{\m}$ in $\mathcal{P}_{\m}$ with the property $\mathcal{P}_{\m}=\mathrm{span}\{\phi_{\m}\circ k: k \in \mathbb K\}.$ This is called as elementary spherical function and given by $$\phi_{\m}(z)=\frac{(\frac{d}{r})_{\m}}{d_{\m}}K_{\m}(z,e), \;\;\text{and}\;\; \|\phi_{\m}\|_{\mathcal{F}}^2= \frac{(\frac{d}{r})_{\m}}{d_{\m}},$$ where $d_{\m}$ is the dimension of the space $\mathcal{P}_{\m}$ (see \cite[Lemma 4.3]{AZ2003}).
Thus, by \eqref{inner-product relation}, we get $$\|\phi_{\m}\|_{\alpha}^2=\frac{1}{\alpha_{\m}}\frac{(d/r)_{\m}}{d_{\m}}.$$

For any function $f$ on $\overline{\Omega},$ its $\mathbb K$-invariantisation, denoted by $f^{\mathbb K},$ is given by $$f^{\mathbb K}(z)= \int_{\mathbb K} f( k\cdot z)dk,$$ where $dk$ is the Haar measure on the compact group $\mathbb K.$ The following result can be found in \cite[Lemma 3.3]{FK1990} and \cite[Proposition 2.9]{BM1996}.
\begin{lemma} \label{invariantization}
 For any signature $\m$ and $z=\sum_{i=1}^r t_ie_i$, we have $$\int_{\mathbb K} |\phi_{\m}(k\cdot z)|^2 dk=\frac{1}{d_{\m}} \phi_{\m}\left(\sum_{i=1}^r t_i^2 e_i\right).$$ 
\end{lemma}

For each $q \geq 0$ and signature $\m=(m_1,\ldots,m_r),$ let $S_{\m}^q$ be the Selberg-Jack symmetric polynomial \cite{KD1997}. Note that the elementary spherical function $\phi_{\m}$ can be expressed as follows: \beq \label{conversion to sb} \phi_{\m}\left(\sum_{i=1}^r x_ie_i\right)&=&\frac{S_{\m}^{a/2}(x_1,\ldots, x_r)}{S_{\m}^{a/2} (1,\ldots,1)}\notag\\
 &=&\frac{f_r^{a/2}(\underline 0)}{f_r^{a/2}(\m)}S_{\m}^{a/2}(x_1,\ldots,x_r),\eeq where $f_r^q(\m)=\prod_{1 \leq i <j \leq r} (m_i-m_j+q(j-i))_{q}$ (see \cite[equation 3.9]{Up2021}). 

 We now establish a criterion for the subnormality of $\boldsymbol M_{z}^{(\alpha)}$ defined on $\mathcal{H}^\alpha(\Omega).$
 \begin{theorem}\label{Main theorem 1}
     The following statements are equivalent:
     \begin{itemize}
         \item [(i)] The multiplication operator tuple $\boldsymbol{M}_z^{(\alpha)}$ on $\mathcal{H}^\alpha(\Omega)$ is a subnormal contractive $d$-tuple.
         \item [(ii)] There is a $\mathbb K$-invariant probability measure $\mu$ supported on $\overline{\Omega}$ such that for all signature $\m$, 
         \beq \label{eq 2} \int_{\overline{\Omega}} |\phi_{\m}|^2 d\mu=\frac{1}{\alpha_{\m}}\frac{(d/r)_{\m}}{d_{\m}}. \eeq 
         \item [(iii)] There is a probability measure $\nu$ supported on $\Delta_r$ such that 
             \beq \label{eq 3} \int_{\Delta_r} S_{\m}^{a/2}(x) d\nu(x)= \frac{f_r^{a/2}(\m)}{f_r^{a/2}(\underline 0)}\frac{(\frac{d}{r})_{\m}}{\alpha_{\m}}.\eeq
     \end{itemize}
 \end{theorem}
 \begin{proof} We first prove the equivalence of $(i)$ and $(ii).$ Assume that the multiplication operator tuple $\boldsymbol M_{z}^{(\alpha)}$ is subnormal contractive $d$-tuple. Then by Proposition \ref{spectrumcontraction}, the Taylor joint spectrum $\sigma(\boldsymbol M_z^{(\alpha)})$ of $\boldsymbol M_z^{(\alpha)}$ is $\overline{\Omega}.$
 By \cite[Theorem 5.1]{BM1996}, there is a unique probability measure $\mu$ supported on $\overline{\Omega}$ such that the inner product on $\mathcal{H}^\alpha(\Omega)$ is given by $$\inp{f}{g}_{\alpha}=\int_{\overline{\Omega}} f\bar{g}d\mu, \qquad  f,g\in \mathcal{H}^\alpha(\Omega).$$ 
  Since the inner product on $\mathcal{H}^\alpha(\Omega)$ is $\mathbb K$-invariant, by the Stone-Weierstrass theorem, the measure $\mu$ is $\mathbb K$-invariant. 
 Furthermore, for any signature $\m$, $$\int_{\overline{\Omega}} |\phi_{\m}|^2 d\mu=\|\phi_{\m}\|^2_{\alpha}=\frac{1}{\alpha_{\m}}\frac{\left(\frac{d}{r}\right)_{\m}}{d_{\m}}.$$
 
 Conversely, assume that there is a $\mathbb K$-invariant probability measure $\mu$ supported on $\overline{\Omega}$ which satisfies \eqref{eq 2}. We equipped $\mathcal P_{\m}$ with the inner product $\inp{p}{q}_{\mu}=\int_{\overline{\Omega}} p \bar{q} d\mu$ for any $p, q \in \mathcal{P}_{\m}.$ Since $\mu$ is $\mathbb K$-invariant, the inner product $\inp{\cdot}{\cdot}_{\mu}$ is $\mathbb K$-invariant as well.
 By Schur Lemma, the $\mathbb K$-invariant inner product on each irreducible subspace $\mathcal{P}_{\m}$ is unique up to a scalar multiple, that is, there is a constant $c_{\m}$ such that $$\frac{1}{\alpha_{\m}}\frac{(d/r)}{d_{\m}}=\inp{\phi_{\m}}{\phi_{\m}}_{\mu}=c_{\m}\inp{\phi_{\m}}{\phi_{\m}}_{\alpha}=c_{\m}\frac{1}{\alpha_{\m}}\frac{(d/r)}{d_{\m}},$$ this yields $c_{\m}=1$ for all signature $\m.$
 Since, $\mathcal{H}^\alpha(\Omega)=\bigoplus_{\m} \mathcal{P}_{\m},$ by \cite[Theorem 5.1]{BM1996}, the multiplication tuple $\boldsymbol M_z^{(\alpha)}$ is subnormal.  
 For any $1\leq \ell \leq r,$ and $p\in \mathcal{P}_{\m},$
\beqn \left\langle \Delta^{(\ell)}(z,w)(\boldsymbol{M}_z^{(\alpha)},\boldsymbol{M}_z^{(\alpha)*})p, p \right\rangle_{\alpha} &=&\prod_{j=1}^{\ell} \left(1+\frac{a}{2}(j-1)\right)\sum_{\beta} \left\langle M_{\psi_{\beta}^{(\ell)}}^{\alpha*} M_{\psi_{\beta}^{(\ell)}}^{\alpha} p, p \right\rangle_{\alpha}\\
%&=& \prod_{j=1}^{\ell} \left(1+\frac{a}{2}(j-1)\right)\sum_{\beta} \left\langle \psi_{\beta}^{(\ell)} p, \psi_{\beta}^{\ell}p \right\rangle_{\alpha}\\
&=& \prod_{j=1}^{\ell} \left(1+\frac{a}{2}(j-1)\right)\sum_{\beta} \int_{\overline{ \Omega}} |\psi_{\beta}^{(\ell)}(z)|^2 |p(z)|^2 d\mu(z)\\
&\overset{\eqref{delk}}{=}&  \int_{\overline{\Omega}} \Delta^{(\ell)}(z,z) |p(z)|^2 d\mu(z). \eeqn
Consequently, by Lemma \ref{domain description}, for any $1 \leq k \leq r$, signature $\m$ and $p\in \mathcal{P}_{\m},$ 
\begin{align*}
    & \sum_{\ell=0}^k (-1)^\ell \binom{r-\ell}{k-\ell}\left\langle \Delta^{(\ell)}(z,w)(\boldsymbol{M}_z^{(\alpha)},\boldsymbol{M}_z^{(\alpha)*})p, p \right\rangle_{\alpha}\\
    &=\sum_{\ell=0}^k (-1)^\ell \binom{r-\ell}{k-\ell}\int_{\overline{\Omega}} \Delta^{(\ell)}(z,z) |p(z)|^2 d\mu(z) \geq 0
\end{align*}
 This proves that $\boldsymbol M_z^{(\alpha)}$ is a subnormal contractive $d$-tuple.

 To prove the equivalence of $(ii)$ and $(iii)$, we first observe that there is a one-to-one correspondence between the class of $\mathbb K$-invariant probability measures $\mu$ supported on $\overline{\Omega}$ and the class of probability measures $\nu$ supported on $\Delta_r = \{(x_1,\ldots,x_r): 0 \leq x_r \leq \cdots \leq x_1 \leq 1\}.$
 Let $\mu$ be a $\mathbb K$-invariant probability measure supported on $\overline{\Omega}.$ Recall that the set $\Delta_r$ is identified as the orbit space $\overline{\Omega}/\mathbb K$ via the quotient map $\pi: \overline{\Omega}\to \Delta_r\approx \overline{\Omega}/{\mathbb K},$ which maps every element $z\in \overline{\Omega}$ to its (unique) singular values $(x_1,x_2,\ldots,x_r)$ arranged in decreasing order.
 Then the measure $\nu$ is nothing but the push forward measure given by $\nu = \mu \circ \pi^{-1}$. If $\nu$ is a probability measure on $\Delta_r$, then the measure $\mu$ can be defined as \beq \label{eq 14}\mu(A):=\int_{\Delta_r}\left(\int_{\mathbb K} \chi_{k^{-1}\cdot A}(x_1e_1+\cdots+x_re_r)dk\right)d\nu(x)\eeq  for every Borel subset $A \subseteq \overline{\Omega}.$  
Furthermore, for any $f\in L^1(\mu)$, \beq \label{eq 13} \int_{\overline{\Omega}} f d\mu=\int_{\Delta_r}\left(\int_{\mathbb K} f (k\cdot\sum_{i=1}^r x_ie_i)dk\right)d\nu(x_1,\ldots,x_r).\eeq 
 Assume $(ii)$ holds true and $\nu$ be the corresponding measure supported on $\Delta_r.$ Then,
 \beqn \frac{1}{\alpha_{\m}} \frac{(\frac{d}{r})_{\m}}{d_{\m}}&\overset{\eqref{eq 2}}{=}&\int_{\overline{\Omega}} |\phi_{\m}|^2d\mu \overset{\eqref{eq 13}}{=} \int_{\Delta_r} \left(\int_{\mathbb K} |\phi (k\cdot\sum_{i=1}^r x_ie_i)|^2dk\right)d\nu(x_1,\ldots,x_r) \\
 &=& \frac{1}{d_{\m}}\int_{\Delta_r} \phi_{\m}(x_1^2 e_1+\cdots+x_r^2 e_r)d\nu(x).\eeqn Here, the last equality follows from Lemma \ref{invariantization}.
 
 Consequently, \eqref{conversion to sb} gives 
    \beqn \int_{\Delta_r} S_{\m}^{a/2}(x_1^2,\ldots,x_r^2) d\nu(x)= \frac{f_r^{a/2}(\m)}{f_r^{a/2}(\underline 0)}\frac{(\frac{d}{r})_{\m}}{\alpha_{\m}}.\eeqn
Now, $(iii)$ follows from a simple change of variable argument.

Conversely, if $\nu$ is the measure given in $(iii)$, then the measure $\mu$ defined by \eqref{eq 14} is a $\mathbb K$-invariant probability measure supported on $\overline{\Omega}.$ By retracing the preceding computation in the reverse direction, \eqref{eq 2} follows immediately. This completes the proof.
 \end{proof}

\begin{remark}
    If $\mu$ is the $\mathbb K$-invariant probability measure on $\overline{\Omega}$ and $\nu$ is the corresponding probability measure on $\Delta_r$ then $\mu$ is supported on $\Omega$ or $\overline{\partial\Omega_j}$ for $j=1,\ldots,r$ if and only if $\nu$ is supported on $\Delta_r^0=\{(x_1,\ldots,x_r): 0 \leq x_r \leq \cdots \leq x_1 <1\}$ or $\Delta_r^j=\{(x_1,\ldots,x_r):0 \leq x_r \leq \cdots \leq x_j=1=x_{j-1}=\ldots=x_1\}$ for $j=1,\ldots,r,$ respectively.
\end{remark}
The part $(iii)$ of Theorem \eqref{Main theorem 1} suggest to consider the following problem.
  \begin{question} 
      Let $q >0$ and $\{b_{\m}\}$ be a sequence of positive numbers indexed over all signatures, with $b_{\underline 0}=1.$ Does there exists a probability measure $\nu$ supported on $\Delta_r$ such that for each signature $\m$, $$\int_{\Delta_r} S_{\m}^q(x) d\nu(x)=b_{\m}\;\; ?$$
  \end{question}
If such a measure exists, then it is necessarily unique. Indeed, suppose that $\nu$ and $\nu'$ are two such measures on $\Delta_r$, then $\int_{\Delta_r} R_{\m}(x) d\nu(x)=\int_{\Delta_r} R_{\m}(x) d\nu'(x)$ for all signature $\m$, where $R_{\m}(x)$ is the symmetric monomials (see \eqref{symmet-selberg}). Since the symmetric monomials separate points on $\Delta_r$, the Stone-Weierstrass theorem implies that $\nu=\nu'.$ The sequence $\{b_{\m}\}_{\m\in \Vec{\mathbb N^r}}$ with this property is called as the {\it twisted moment sequence} with parameter $q$. The measure $\nu$ will be referred to as the representing measure of $\{b_{\m}\}.$
In the remainder of this section, we establish a relationship between twisted moment sequence and a certain Hausdorff moment sequence.

 \subsection{A twisted moment problem}  
   For any signature $\m=(m_1,\ldots,m_r)$, let $|\m|=m_1+\cdots+m_r$ denote its norm, and define the length $l(\m)$ to be the cardinality of the set $\{j: m_j>0\}.$ In our setting, $l(\m) \leq r.$
   
   We consider two types of partial orders on the set of all signatures with fixed norm. The first is the {\it reverse lexicographic ordering}, denoted by  $\leq^R,$ defined as follows: for any two signatures $\m_1=(m_1^1,\ldots,m_r^1)$ and $\m_2=(m_1^2,\ldots,m_r^2)$ with same norm, we say $\m_1 \leq^R \m_2$ if either $\m_1=\m_2$ or the first non-vanishing difference $m^2_i-m_i^1$ is positive for some $1\leq i \leq r.$ This is a complete ordering on the set of all signatures with fixed norm. For instance, if norm of the signatures is $5$ and length is less than equal to $3$, then the ordering is $(5,0,0),\; (4,1,0),\; (3,2,0),\; (3,1,1),\; (2,2,1).$
    The second is the {\it dominance partial ordering}, also known as the natural partial ordering, defined as follows: for any two signatures $\m_1$ and $\m_2$ with fixed norm, we say $\m_1 \leq \m_2$, if $m_1^1+\cdots+m_i^1 \leq m_1^2+\cdots+m_i^2$ for all $1 \leq i \leq r$. If $\m_1 \leq \m_2$ and $\m_1 \neq \m_2$ then we write $\m_1 <\m_2.$ Note that this ordering is not total in general (see \cite[p. 6]{MD1995}). Moreover, the dominance partial ordering is compatible with the reverse lexicographic ordering.  
  
  A matrix $(A_{\m,\s})$ indexed by all signatures of fixed norm in the reverse lexicographic ordering is said to be strictly upper triangular if $A_{\m,\s}=0$ unless $\s \leq \m$ and is said to be strictly upper unitriangular if in addition $A_{\m,\m}=1$ for all $\m.$ The set of all strictly upper unitriangular matrices forms a group (see \cite[Chapter I-6]{MD1995}). Here, we adopt the convention that the rows and columns in the matrix $(A_{\m,\s})$ are indexed in decreasing reverse lexicographic order, from top to bottom and from left to right, respectively.
  
  Let $x=(x_1,\ldots,x_r)$ be the tuple of $r$ independent variables $x_1,\ldots,x_r,$ then for each $\m=(m_1,\ldots,m_r),$ the symmetric monomials is given by $R_{\m}(x)=\sum_{\beta} x^\beta,$ where the summation is taken over all distinct permutation $\beta $ of $\m.$
  It follows from \cite[ Chapter VI-10.13]{MD1995} that the Selberg-Jack symmetric polynomials can be written as: $$S_{\m}^q (x_1,\ldots,x_r)=R_{\m}(x)+\sum_{\s < \m} K_{\m,\s}(q) R_{\s}(x).$$ 
  The numbers $K_{\m, \s}(q)$, coefficients of $R_{\s}$ in the above expression, are rational functions in $q,$ known as the Kostka polynomials. The matrix $K(q):=(K_{\m,\s}(q))$ indexed by all signatures of fixed norm in the reverse lexicographic ordering is strictly upper unitriangular. 
  Therefore, the inverse of $K(q)$ is also strictly upper unitriangular. The $(\m,\s)th$ entry in the inverse matrix is denoted by $K^{-1}_{\m, \s}(q).$
 Thus, for any signature $\m$, \beq \label{symmet-selberg}R_{\m}(x)=S_{\m}^q(x)+\sum_{\s < \m} K^{-1}_{\m, \s}(q)S_{\s}^q(x).\eeq
 
The case when $q=1$, the Selberg-Jack polynomial coincides with the Schur polynomial $S_{\m}^1(x)$, which in general is denoted by $S_{\m}$ (see \cite{KD1988}). The matrix $K(1)$ is called the Kostka matrix. A detailed computation of entries in the Kostka matrix and its inverse can be found in \cite{ER1990}, \cite{HD2003}.

 For $r=2,$ it follows from equations (1.7) and (1.4) of \cite{KD1997} that the Selberg–Jack symmetric function for the signature $\m=(m_1,m_2)$ is given by $$S_{\m}^q(x_1,x_2)=\sum_{\s \leq \m}K_{\m, \s}(q) R_{\s}(x_1,x_2),$$ where \beq \label{rank 2 Kostka}K_{\m, \s}(q)=\frac{(m_2-m_1)_{(m_1-s_1)} (q)_{(m_1-s_1)}}{(m_1-s_1)!\;(1-m_1+m_2-q)_{(m_1-s_1)}}, \qquad \s=(s_1, s_2).\eeq 
In addition, if $q=1$ then $K_{\m, \s}=1$ for all $\s \leq \m.$ Therefore, in this case, the inverse Kostka matrix consists of $1$’s on the diagonal, $-1$’s on the principal superdiagonal, and $0$ elsewhere.

 For any signature $\m$, let $\omega(\m)$ denote the order of the isotropy group $\{\sigma\in \mathrm{Sym}(r): m_{\sigma(k)}=m_k,\; 1 \leq k \leq r\}.$ Then $$R_{\m}(x)=\frac{1}{\omega(\m)}\sum_{\sigma \in \mathrm{Sym}(r)}\prod_{k=1}^r x_{\sigma(k)}^{m_k}.$$
 Let $\nu$ be the representing measure of a twisted moment sequence $\{b_{\m}\}$ with the parameter $q$, then
 \beq \label{eq 44} \sum_{\s \leq \m} K^{-1}_{\m, \s}(q) b_{\s} &=& \sum_{\s \leq \m} K^{-1}_{\m, \s}(q)\int_{\Delta_r} S_{\s}^q (x)d\nu(x)\notag\\
 &=& \int_{\Delta_r} R_{\m}(x)d\nu(x)\notag\\
 &=& \frac{1}{\omega(\m) }\int_{\Delta_r}\sum_{\sigma \in \mathrm{Sym}(r)}\prod_{k=1}^r x_{\sigma(k)}^{m_k} d\nu(x).\eeq
 
 From a given twisted moment sequence $\{b_{\m}\}$, we define a sequence $\{c_{m}(q)\}_{m\in \mathbb Z_+^r}$ (indexed over $r$-tuple of non-negative integers) as follows: for any $r$-tuple $m=(m_1,\ldots,m_r) \in \mathbb Z_+^r,$ \beq \label{relation between two sequence} c_{m}(q):=\omega(\m)\sum_{\s \leq \m} K^{-1}_{\m, \s}(q) b_{\s},\eeq where $\m$ is the signature obtained from $(m_1,\ldots,m_r)$ by rearranging its entries in decreasing order. 
 Note that $c_0(q)=r!$ and $\{c_m(q)\}_{m\in \mathbb Z_+^r}$ is invariant under the symmetric group, that is, $c_{(m_1,\ldots, m_r)}=c_{(m_{\sigma(1)}, \ldots, m_{\sigma(r)})}$ for all $\sigma\in \mathrm{Sym}(r).$

  \begin{theorem} \label{moment problem condition}
      A sequence $\{b_{\m}\}$ is a twisted moment sequence with the parameter $q$ if and only if the sequence $\{c_{m}(q)\}_{m \in \mathbb Z_+^r}$ given by \eqref{relation between two sequence} is a Hausdorff moment sequence.
  \end{theorem}
\begin{proof}
    Assume that there exists a probability measure $\nu$ supported on $\Delta_r$ such 
    that for any signature $\m$, $$b_{\m}=\int_{\Delta_r} S_{\m}^q(x) d\nu(x).$$ By \eqref{eq 44}, $c_m(q)$ must satisfies  \beq \label{eq 15}c_m(q)=\int_{\Delta_r}\sum_{\sigma \in \mathrm{Sym}(r)}\prod_{k=1}^r x_{\sigma(k)}^{m_k} d\nu(x),\;\; m \in \mathbb Z_+^r.\eeq 
    %where $\m$ is the unique signature obtained from $m$ by rearranging its term in decreasing order.
    
    For each $\sigma\in \mathrm{Sym}(r)$, consider the bijective map $\pi_\sigma: [0,1]^r\to [0,1]^r$ given by $\pi_\sigma(x_1,\ldots,x_r)=(x_{\sigma(1)},\ldots,x_{\sigma(r)}).$ Then $\Delta_r^\sigma:=\pi_\sigma(\Delta_r)=\{(x_1,\ldots,x_r)\in [0,1]^r: 0 \leq x_{\sigma^{-1}(r)} \leq \cdots \leq x_{\sigma^{-1}(1)}\leq 1\}.$ Let $\nu^\sigma$ be the measure on $\pi_\sigma(\Delta_r),$ defined as $\nu^\sigma=\nu\circ\pi_\sigma^{-1}.$ 
    For any Borel subset $A$ of $[0,1]^r,$ consider the measure $\widetilde{\nu}$ as $$\widetilde{\nu}(A):=\sum_{\sigma\in \mathrm{Sym}(r)} \nu^\sigma\left(A \cap\Delta_r^\sigma\right).$$ Then the measure $\widetilde{\nu}$ is invariant under the symmetric group $\mathrm{Sym}(r).$
    Thus, for any $m=(m_1,\ldots,m_r)\in \mathbb Z_+^r$, 
    \beqn \int_{[0,1]^r} x_1^{m_1}\cdots x_r^{m_r} d\widetilde{\nu}(x) 
    %&=& \sum_{\sigma\in \mathrm{Sym}(r)} \int_{[0,1]^r} x_1^{m_1}\cdots x_r^{m_r} d\nu^\sigma(x)\\
    &=& \sum_{\sigma\in \mathrm{Sym}(r)} \int_{\Delta_r^\sigma} x_1^{m_1}\cdots x_r^{m_r} d\nu\circ \pi_\sigma^{-1}(x)\\
    &=& \sum_{\sigma\in \mathrm{Sym}(r)} \int_{\Delta_r} x_{\sigma(1)}^{m_1}\cdots x_{\sigma(r)}^{m_r}d\nu(x)
    \overset{\eqref{eq 15}}= c_m(q).\eeqn
    This shows that $\{c_m(q)\}_{m\in \mathbb Z_+^r}$ is a Hausdorff moment sequence with the $\mathrm{Sym}(r)$-invariant representing measure.

    Conversely, assume that $\{c_m(q)\}_{m \in \mathbb Z_+^r}$ obtained in \eqref{relation between two sequence} is a Hausdorff moment sequence with the representing measure $\widetilde{\nu}.$ Note that the measure $\widetilde{\nu}$ is invariant under the symmetric group $\mathrm{Sym}(r).$ 
    Let $x=(x_1,\ldots,x_r)\in [0,1]^r$, then there is a unique tuple $y=(y_1,\ldots,y_r)$ obtained by permuting the entries of $x$ such that $0 \leq y_r \leq \cdots \leq y_1\leq 1.$ Consider an onto map $P: [0,1]^r \to \Delta_r$ given by $P(x)=y.$
    For every Borel subset $A$ of $\Delta_r$, we define a probability measure $\nu$ as follows: $$\nu(A)=\frac{1}{r!} \widetilde{\nu}\circ P^{-1}(A).$$ 
    Since the measure $\widetilde{\nu}$ is $\mathrm{Sym}(r)$-invariant, it is easy to verify that 
    \beqn c_{m}(q)=  \int_{[0,1]^r} x_1^{m_1}\cdots x_r^{m_r}d\widetilde{\nu}(x) =  \sum_{\sigma\in \mathrm{Sym}(r)} \int_{\Delta_r} x_{\sigma(1)}^{m_1}\cdots x_{\sigma(r)}^{m_r} d\nu(x).\eeqn 
    Therefore, by \eqref{relation between two sequence}, for any signature $\m$,  $$\int_{\Delta_r}R_{\m}(x) d\nu(x)=\sum_{\s \leq \m} K^{-1}_{\m, \s}(q) b_{\s}.$$ 
    Thus, $$\int_{\Delta_r} S_{\m}^q(x) d\nu(x)=\sum_{\s \leq \m}K_{\m,\s}(q) \int_{\Delta_r}R_{\s}(x) d\nu(x)= b_{\m}.$$ This completes the proof.
   \end{proof} 

  A direct consequence of Theorem \ref{Main theorem 1}, along with Theorem \ref{moment problem condition}, provides a characterization of subnormality of the tuple $\boldsymbol M_z^{(\alpha)}$ in terms of $\{\alpha_{\m}\} .$
   \begin{corollary}\label{complete characterization}
   The following statements are equivalent.
   \begin{itemize}
    \item [(i)]  $\boldsymbol M_z^{(\alpha)}$ is subnormal and $\sigma(\boldsymbol M_z^{(\alpha)})\subseteq \overline{\Omega}$.
       \item [(ii)] $\boldsymbol{M}_z^{(\alpha)}$ is a subnormal contractive $d$-tuple.
       \item [(iii)] The sequence $\{b_{\m}\}$  given by $$b_{\m}=\frac{f_r^{a/2}(\m)}{f_r^{a/2}(\underline 0)}\frac{(\frac{d}{r})_{\m}}{\alpha_{\m}},$$ is a twisted moment sequence with parameter $\frac{a}{2}.$
       \item  [(iv)] The sequence $\{c_m\}_{m\in \mathbb Z_+^r}$ given by $$c_m=\omega(\m)\sum_{\s \leq \m} K^{-1}_{\m, \s}(a/2) \frac{f_r^{a/2}(\s)}{f_r^{a/2}(\underline 0)}\frac{(\frac{d}{r})_{\s}}{\alpha_{\s}}, $$
       is a Hausdorff moment sequence.
   \end{itemize}
   \end{corollary}

\section{Applications}\label{S4}
Let $\boldsymbol M_z^{(\nu)}$ be the operator tuple of multiplications by the coordinate functions on $\mathcal{H}^\nu(\Omega).$ 
Suppose that $\boldsymbol M_z^{(\nu)}$ is subnormal. Since $\boldsymbol M_z^{(\nu)} $ is necessarily bounded, it follows that $\sigma(\boldsymbol{M}_z^{(\nu)})=\overline{\Omega}$ (see \cite[Theorem 4.1]{AZ2003}).
By the equivalence of conditions $(i)$ and $(ii)$ in Corollary \ref{complete characterization}, $\boldsymbol M_z^{(\nu)}$, in particular, is a contractive $d$-tuple. Therefore, by the necessary part of Theorem \ref{contractionpoint}, $$\nu\in \left\{\frac{d} {r},\; \frac{d}{r}+\frac{a}{2},\ldots, \frac{d}{r}+\frac{a}{2}(r-1)\right\}\bigcup\left(\frac{d}{r}+\frac{a}{2}(r-1),\infty\right).$$ 
The main objective of this section is to prove that $\boldsymbol M_z^{(\nu)}$ is subnormal precisely for the above values of $\nu.$ This result is not new, first proved in \cite[Theorem 1.3]{BM1996} for type-I domain and later in \cite[Theorem 6.1]{AZ2003} for any classical Cartan domain.
\begin{theorem}\label{subnormality alternative}
The tuple $\boldsymbol M_z^{(\nu)}$ is subnormal if and only if $\nu=\frac{d}{r}, \frac{d}{r}+\frac{a}{2}, \ldots, \frac{d}{r}+\frac{a}{2}(r-1)$ or $v >\frac{d}{r}+\frac{a}{2}(r-1).$
\end{theorem}
In view of Theorem \ref{Main theorem 1}(iii), we show that $\left\{\frac{f_r^{a/2}(\m)}{f_r^{a/2}(0)} \frac{\left(d/r\right)_{\m}}{(\nu)_{\m}}\right\}$ is a twisted moment sequence with the parameter $\frac{a}{2}$ for the values of $\nu$ specified in Theorem \ref{subnormality alternative}. The proof is divided into two parts: one for the continuous values and the other for the discrete values of $\nu$. For $\nu>\frac{d}{r}+\frac{a}{2}(r-1),$ the representing measure is given by pushforward of the measure $\Delta(z,z)^{\nu-p}dm(z)$ under the map $\pi: \overline{\Omega} \to \Delta_r,$ up to a suitable normalization constant, which also appears in the proof of \cite[Theorem 6.1]{AZ2003}.

For the discrete values of $\nu,$ our argument differs from that of \cite[Theorem 6.1]{AZ2003}.  In this case, we construct the representing measures by adapting a technique used in proving the equivalence of $(i)$ and $(ii)$ in Theorem 1.3 of \cite{BM1996}, where a key ingredient is a decomposition of Schur polynomials obtained in \cite[Proposition 2.4]{BM1996}. However, to the best of our knowledge, an analogous decomposition is not known for the Selberg–Jack symmetric functions, and it is unclear whether such a decomposition holds. To circumvent this difficulty, we establish an identity in Proposition \ref{important lemma} involving Selberg-Jack symmetric functions, which appears to be new and may be of independent interest. This identity plays a crucial role in the proof of Theorem \ref{subnormality alternative}.

\begin{proof}[Proof of Theorem \ref{subnormality alternative}(continuous part)]
If $\nu >a(r-1)+b+1=\frac{d}{r}+\frac{a}{2}(r-1)$ and $p_r=a(r-1)+b+2$, then
consider the measure
\beq \label{measure for cont} d\mu_r^\nu(t) = C_\nu(r)\prod_{i=1}^r t_i^b (1-t_i)^{\nu-p_r}\prod_{1 \leq i <j \leq r} |t_i-t_j|^{a}dt_1 \cdots dt_r\\
\text{where }\; C_\nu(r)=\frac{1}{f_r^{a/2}(0)}\prod_{i=1}^r \frac{\Gamma(\nu-\frac{a}{2}(i-1))}{\Gamma(\frac{a}{2}(r-i)+b+1)\Gamma(\nu-p_r+1+\frac{a}{2}(r-i))}.\notag\eeq 
Then, by a generalized Selberg's integral formula given in \cite[Theorem 1]{KD1997}, 
\beqn \int_{\Delta_r} S_{\m}^{a/2}(t) d\mu_r^{\nu}(t) &=& \frac{1}{r!} C_\nu(r) \int_{[0,1]^r} S_{\m}^{a/2}(t) \prod_{i=1}^r t_i^b (1-t_i)^{\nu-p_r}\prod_{1 \leq i <j \leq r} |t_i-t_j|^{a}dt_1 \cdots dt_r\\
&=& \frac{f_r^{a/2}(\m)}{f_r^{a/2}(0)} \prod_{i=1}^r  \frac{\Gamma(\nu-\frac{a}{2}(i-1))}{\Gamma(\frac{a}{2}(r-i)+b+1)\Gamma(\nu-p_r+1+\frac{a}{2}(r-i))}\\
&& \times \prod_{i=1}^r \frac{\Gamma(b+1+\frac{a}{2}(r-i)+m_i) \Gamma(\nu-p_r+1+\frac{a}{2}(r-i))}{\Gamma(b+1+\nu-p_r+1+\frac{a}{2}(2r-i-1)+m_i)}\\
&=& \frac{f_r^{a/2}(\m)}{f_r^{a/2}(0)} \prod_{i=1}^r \frac{\Gamma(b+1+\frac{a}{2}(r-i)+m_i) \Gamma(\nu-\frac{a}{2}(i-1))}{\Gamma(b+1+\frac{a}{2}(r-i))\Gamma(\nu-\frac{a}{2}(i-1)+m_i)}\\
&=& \frac{f_r^{a/2}(\m)}{f_r^{a/2}(0)} \frac{(b+1+\frac{a}{2}(r-1))_{\m}}{(\nu)_{\m}}=\frac{f_r^{a/2}(\m)}{f_r^{a/2}(0)} \frac{(d/r)_{\m}}{(\nu)_{\m}}.\eeqn
This shows that $\boldsymbol M_z^{(\nu)}$ is subnormal for $\nu> \frac{d}{r}+\frac{a}{2}(r-1).$ 
\end{proof}

For the discrete values of $\nu$, we first establish the following identity involving Selberg-Jack symmetric functions. The proof of this identity relies on known results involving Selberg-Jack symmetric functions in infinite variables, whereas we have primarily worked with symmetric functions in finitely many variables. To seamlessly utilize the existing literature on Selberg–Jack symmetric functions, it is convenient to reset a few notations. 
We consider signatures to be partitions $\m=(m_1,m_2,\ldots)$ consisting of finitely many positive integers arranged in weakly decreasing order. In \cite{MD1995}, the Selberg-Jack symmetric function is denoted by $P_{\m}^{(\alpha)}(x)$. In our setup $\alpha=\frac{1}{q},$ we find it convenient to use the notation $P_{\m}\left(x;\frac{1}{q}\right)$ for the Selberg-Jack symmetric function in the infinite variables $x=(x_1,x_2, \ldots)$ with parameter $q>0$.
For any positive integer $n,$ if $x_1=x_2=\cdots=x_n=1$ and $x_{n+1}=x_{n+2}=\cdots=0$, then we write $P_{\m}\left(x;\frac{1}{q}\right)$ as $P_{\m}\left(1^n;\frac{1}{q}\right).$ Note that if the number of nonzero entries $l(\m)$ in $\m$, satisfies $l(\m)>n,$ then $P_{\m}\left(1^n;\frac{1}{q}\right)=0.$ Moreover, when we work with finitely many variables $(x_1,\ldots, x_r)$, we regard this as a particular case of the infinite sequence $x=(x_1,x_2,\ldots)$ by setting $x_{r+1}=x_{r+2}=\cdots=0.$ 
With this understanding, $S_{\m}^q(x_1,\ldots,x_r)=P_{\m}\left(x;\frac{1}{q}\right)=0$ for any partition $\m$ with $l(\m)>r.$
 
For any partition $\m=(m_1,m_2,\ldots)$, we denote $\m'=(m_1',m_2',\ldots)$ as the conjugate of $\m,$ where $m_i'$ denotes the cardinality of the set $ \{j : m_j \geq i\}.$
 The diagram of a partition is defined as the set of points $(i,j)\in \mathbb Z_+^2$ such that $1\leq j \leq m_i.$ 
In drawing such diagrams, we shall adopt the convention, as the matrices, that the first coordinate $i$ (the row index) increases as one goes downwards, and the second coordinate $j$ (the column index) increases as one goes from left to right.
Note that the generalized Pochhammer symbol given above is a product of generalized rising factorial with step $\frac{a}{2}.$ One can replace $\frac{a}{2}$ by any $q>0$, and write the generalized Pochhammer symbol by the formula $$(c)_{\m}^q=\prod_{j=1}^r \left(c-q(j-1)\right)_{m_j}.$$ We drop the superscript $q$ whenever $q=\frac{a}{2}.$
\begin{proposition}\label{important lemma}
    For any positive integers $n$ and $\ell$, partition $\m$, and $c\in \mathbb C,$ $$\sum_{\s \subseteq \m} P_{\m/\s}\left(1^n;\frac{1}{q}\right)  P_{\s}\left(1^\ell;\frac{1}{q}\right) \frac{(c)^q_{\s}}{(c+q(n+\ell))^q_{\s}}=P_{\m}\left(1^{n+\ell};\frac{1}{q}\right)\frac{(c+qn)^q_{\m}}{(c+q(n+\ell))^q_{\m}},$$ where $\s\subseteq\m$ means that $\s$ is a sub-partition of $\m$, that is, $s_i \leq m_i,$ and $\m/\s$ is called corresponding skew partition. 
\end{proposition}
\begin{proof} 
 Let $x=(x_1,x_2,\ldots)$ and $y=(y_1,y_2,\ldots)$ be two set of indeterminants. Then, by \cite[equation(SC), p. 64]{KD1997}, \beqn \sum_{\m} b_{\m}\left(\frac{1}{q}\right)P_{\m}\left(x; \frac{1}{q}\right)P_{\m}\left(y; \frac{1}{q}\right)=\prod_{i,j\geq 1}(1-x_i y_j)^{-q},\eeqn where $$b_{\m}\left(\frac{1}{q}\right)=\prod_{(i,j)\in \m}\frac{m_i-j+q(m'_j-i)+q}{m_i-j+q(m'_j-i)+1}.$$
    %Recall from \cite[Chapter VI-10.16]{MD1995}, 
    Set $Q_{\m}\left(x;\frac{1}{q}\right)=b_{\m}\left(\frac{1}{q}\right) P_{\m}\left(x; \frac{1}{q}\right),$ then the above equation becomes 
    \beq \label{eq 18}\prod_{i,j\geq 1} (1-x_i y_j)^{-q}=\sum_{\m} Q_{\m}\left(x; \frac{1}{q}\right)P_{\m}\left(y; \frac{1}{q}\right).\eeq
    Furthermore, equation (4.12) in \cite[Chapter VI]{MD1995} and its subsequent discussion yields (see also \cite[Chapter VI-10.16]{MD1995}) \beq \label{eq 26}\Inp{P_{\m}\left(x;\frac{1}{q}\right)}{Q_{\s}\left(x;\frac{1}{q}\right)}=\delta_{\m\; \s}, \;\; \text{ and }\;\;\Inp{P_{\m}\left(x;\frac{1}{q}\right)}{P_{\m}\left(x;\frac{1}{q}\right)}^{-1}=b_{\m}\left(\frac{1}{q}\right).\eeq Here, the inner product $\Inp{\cdot}{\cdot}$ is as given in \cite[Chapter VI-1.4]{MD1995} (see \cite[equation 3]{St1989}).

     For any two partition $\s$ and $\\lambda$, \beq \label{eq 21}Q_{\s}\left(x;\frac{1}{q}\right)Q_{\underline\lambda}\left(x; \frac{1}{q}\right)=\sum_{\m} f^{\m'}_{\s'\; \underline\lambda'}\left(\frac{1}{q}\right)Q_{\m}\left(x; \frac{1}{q}\right),\eeq where the right hand sum is over all such partition $\m$ for which $|\m|=|\s|+|\\lambda|$ and $\m', \s',\\lambda'$ denotes the corresponding conjugate partitions. 
  A description of coefficients $f^{\m'}_{\s'\; \\lambda'}\left(\frac{1}{q}\right)$ can be found \cite[Chapter VI-7.3]{MD1995}. Recall that $\{P_{\\lambda}\left(x; \frac{1}{q}\right)\}$ forms a basis of the space of all symmetric polynomials over the field $\mathbb Q\left(\frac{1}{q}\right).$
  Therefore, $P_{\m/\s}\left(x;\frac{1}{q}\right)$ is a linear combination of all such $P_{\\lambda}\left(x; \frac{1}{q}\right)$ for which $|\\lambda|=|\m|-|\s|$ and given by $P_{\m/\s}\left(x;\frac{1}{q}\right)=\sum_{\\lambda}c_{\\lambda,\m,\s} P_{\\lambda}\left(x; \frac{1}{q}\right).$ 
  Now, we observe that the coefficients $c_{\\lambda,\m,\s}$ are nothing but $f^{\m'}_{\s'\; \\lambda'}\left(\frac{1}{q}\right).$ Indeed, for any partition $\\lambda$ with $|\\lambda|=|\m|-|\s|$, 
  \beqn
      \Inp{ P_{\m/\s}\left(x;\frac{1}{q}\right)} {P_{\underline\lambda}\left(x; \frac{1}{q}\right)} &=& \frac{1}{b_{\underline\lambda}(1/q)} \Inp{ P_{\m/\s}\left(x;\frac{1}{q}\right)}{  Q_{\underline\lambda}\left(x; \frac{1}{q}\right)}\\
      &=& \frac{1}{b_{\underline\lambda}(1/q)} \Inp{ P_{\m}\left(x;\frac{1}{q}\right)}{  Q_{\s}\left(x; \frac{1}{q}\right) Q_{\underline\lambda}\left(x; \frac{1}{q}\right)}\\
      &\overset{\eqref{eq 21}}=& \frac{1}{b_{\underline\lambda}(1/q)} \Inp{ P_{\m}\left(x;\frac{1}{q}\right)}{\sum_{\m} f^{\m'}_{\s'\; \underline\lambda'}\left(\frac{1}{q}\right)Q_{\m}\left(x; \frac{1}{q}\right)}\\
      &\overset{\eqref{eq 26}}=&  \frac{1}{b_{\underline\lambda}(1/q)}f^{\m'}_{\s'\; \underline\lambda'}\left(\frac{1}{q}\right)\overset{\eqref{eq 26}}= f^{\m'}_{\s'\; \underline\lambda'}\left(\frac{1}{q}\right)\Inp{P_{\underline\lambda}\left(x; \frac{1}{q}\right)}{P_{\underline\lambda}\left(x; \frac{1}{q}\right)}.
  \eeqn
The second equality in the above calculation follows from \cite[Chapter VI-(7.6')]{MD1995}.
  Thus, \beq \label{eq 19}P_{\m/\s}\left(x;\frac{1}{q}\right)=\sum_{\underline\lambda} f^{\m'}_{\s'\; \underline\lambda'}\left(\frac{1}{q}\right)P_{\underline\lambda}\left(x; \frac{1}{q}\right),\eeq
where the sum runs over all partitions $\\lambda$ for which $|\\lambda|=|\m|-|\s|.$
   
  By putting $y=1^n$ in \eqref{eq 18} and then multiplying by $Q_{\s}\left(x; \frac{1}{q}\right)$ on both sides, we get 
  \beq \label{eq 22}\prod_{i}(1-x_i)^{-nq}Q_{\s}\left(x; \frac{1}{q}\right)&=& \sum_{\underline\lambda}Q_{\s}\left(x; \frac{1}{q}\right) Q_{\underline\lambda}\left(x; \frac{1}{q}\right)P_{\underline\lambda}\left(1^n; \frac{1}{q}\right) \notag\\
  &\overset{\eqref{eq 21}}=& \sum_{\underline\lambda} \sum_{\m}f^{\m'}_{\s'\; \underline\lambda'}\left(\frac{1}{q}\right) Q_{\m}\left(x; \frac{1}{q}\right)P_{\underline\lambda}\left(1^n; \frac{1}{q}\right)\notag\\
 &=&  \sum_{\m}\left(\sum_{\underline\lambda} f^{\m'}_{\s'\; \underline\lambda'}\left(\frac{1}{q}\right)P_{\underline\lambda}\left(1^n; \frac{1}{q}\right)\right)Q_{\m}\left(x; \frac{1}{q}\right) \notag\\
  &\overset{\eqref{eq 19}}{=}& \sum_{\m} P_{\m/\s}\left(1^n; \frac{1}{q}\right)Q_{\m}\left(x;\frac{1}{q}\right).\eeq
The last sum runs over all partitions $\m$ for which $\s \subseteq \m.$
For any two fixed positive integers $\ell$ and $n$, and a complex number $c$, define \beq \label{eq 20} G_\ell(x):=\sum_{\s} Q_{\s}\left(x; \frac{1}{q}\right)P_{\s}\left(1^\ell; \frac{1}{q}\right)\frac{(c)^q_{\s}}{(c+q(n+\ell))^q_{\s}}.\eeq 
Therefore, 
\begin{align} \label{eq 24}
    &G_\ell(x) \prod_{i}(1-x_i)^{-qn} = \sum_{\s}\left(\prod_{i}(1-x_i)^{-qn}Q_{\s}\left(x; \frac{1}{q}\right)\right)\left(P_{\s}\left(1^\ell; \frac{1}{q}\right)\frac{(c)^q_{\s}}{(c+q(n+\ell))^q_{\s}}\right) \notag\\
&\overset{\eqref{eq 22}}= \sum_{\s}\left(\sum_{\m} P_{\m/\s}\left(1^n; \frac{1}{q}\right)Q_{\m}\left(x;\frac{1}{q}\right)\right)\left(P_{\s}\left(1^\ell; \frac{1}{q}\right)\frac{(c)^q_{\s}}{(c+q(n+\ell))^q_{\s}}\right)\notag\\
&= \sum_{\m} Q_{\m}\left(x; \frac{1}{q}\right)\left(\sum_{\s\subseteq \m} P_{\m/\s}\left(1^n;\frac{1}{q}\right) P_{\s}\left(1^\ell; \frac{1}{q}\right)\frac{(c)^q_{\s}}{(c+q(n+\ell))^q_{\s}}\right)
\end{align}

For any partition $\m$, if $l(\m)\leq n$, then by \cite[equation (S), p. 69]{KD1997} $$P_{\m}\left(1^n; \frac{1}{q}\right)=\prod_{(i,j)\in \m}\frac{nq+j-1-q(i-1)}{m_i-j+q(m'_j-i)+q}.$$ Also, if  $J_{\m}\left(x;\frac{1}{q}\right)$ denote the ``Jack symmetric function", then by \cite[Theorem 5.4]{St1989} $$J_{\m}\left(1^n; \frac{1}{q}\right)=\frac{1}{q^{|\m|}}\prod_{(i,j)\in \m}(nq+j-1-q(i-1)).$$
Since $J_{\m}\left(x;\frac{1}{q}\right)=c_{\m}\left(\frac{1}{q}\right) P_{\m}\left(x;\frac{1}{q}\right)$ (see \cite[Theorem 2]{KD1997}), by taking $x=1^n,$ we get  \beq \label{eq 16} c_{\m}\left(\frac{1}{q}\right)=\frac{1}{q^{|\m|}}\prod_{(i,j)\in \m} (m_i-j+q+q(m_j'-i)).\eeq
We now set \beq \label{eq 23} j_{\m}\left(\frac{1}{q}\right)&:=& 
    \Inp{J_{\m}\left(x;\frac{1}{q}\right)}
    {J_{\m}\left(x;\frac{1}{q}\right)} = \frac{\left(c_{\m}(1/q)\right)^2}{b_{\m}\left(\frac{1}{q}\right)}\notag\\
   &\overset{\eqref{eq 16}}=& \frac{1}{q^{2|\m|}}\prod_{(i,j)\in \m} (m_i-j+q(m_j'-i)+q)(m_i-j+q(m_j'-i)+1).\eeq
The generalised hypergeometric function for three complex numbers $u,v,w$ (see \cite[equation (4)]{YZ1992}) is given by
      $$_2F_1(u,v;w;x)=\sum_{\s} \frac{(u)^q_{\s} (v)^q_{\s}}{(w)^q_{\s}}\frac{J_{\s}\left(x;\frac{1}{q}\right)}{j_{\s}(1/q) q^{|\s|}}.$$
Thus, $G_\ell(x)$ (see \eqref{eq 20}) can be rewritten as 
\beqn G_\ell(x) &=& \sum_{\s}\frac{(c)^q_{\s}}{(c+q(n+\ell))^q_{\s}}b_{\s}(1/q) P_{\s}\left(x; \frac{1}{q}\right)\prod_{(i,j)\in \s}\frac{q\ell+j-1-q(i-1)}{s_i-j+q(s_j'-i)+q}\\
&=&\sum_{\s}\frac{(c)^q_{\s}}{(c+q(n+\ell))^q_{\s}} \prod_{(i,j)\in \s}\frac{s_i-j+q(s_j'-i)+q}{s_i-j+q(s_j'-i)+1}\frac{J_{\s}\left(x; \frac{1}{q}\right) q^{|\s|}}{\prod_{(i,j)\in \s}(s_i-j+q+q(s_j'-i))}\\
&&\times \prod_{(i,j)\in \s}\frac{q\ell+j-1-q(i-1)}{s_i-j+q(s_j'-i)+q}\\
&\overset{\eqref{eq 23}}=& \sum_{\s}\frac{(c)^q_{\s}\; (q\ell)^q_{\s}}{(c+q(n+\ell))^q_{\s}}\frac{J_{s}\left(x; \frac{1}{q}\right)}{q^{|\s|}\;j_{\s}(1/q)}= \;_2F_1 \left(c, q\ell; c+q(n+\ell); x\right). \eeqn
Here, we have used the fact that $$\prod_{(i,j)\in \s}(q\ell+j-1-q(i-1))=\prod_{i=1}^{l(\s)}\prod_{j=1}^{s_i}(q\ell+j-1-q(i-1))=(q\ell)^q_{\s}.$$
Since, by \cite[Proposition 3.2]{YZ1992}, 
$$\prod_{i\geq 1}(1-x_i)^{u+v-w} \;_2F_1 (u,v; w; x)= \;_2F_1(w-u,w-v; w;x),$$
we get 
      \beq \label{eq 10}G_{\ell}(x)\prod_{i\geq 1}(1-x_i)^{-qn}  &=& \; _2F_1(q(n+\ell), c+qn; c+q(n+\ell); x)\notag\\
      &=& \sum_{\m} \frac{(c+qn)^q_{\m} (q(n+\ell)^q_{\m})}{(c+q(n+\ell))^q_{\m}} \frac{J_{\m}\left(x;\frac{1}{q}\right)}{j_{\m}(1/q) q^{|\m|}}\notag\\
      &=& \sum_{\m} \frac{(c+qn)^q_{\m}}{(c+q(n+\ell))^q_{\m}} Q_{\m}\left(x;\frac{1}{q}\right) P_{\m}\left(1^{n+\ell}; \frac{1}{q}\right).\eeq
  Last equality follows from the fact that $$Q_{\m}\left(x;\frac{1}{q}\right) P_{\m}\left(1^{n+\ell}; \frac{1}{q}\right)=\frac{(q(n+\ell))^q_{\m}}{q^{|\m|} j_{\m}(1/q)} J_{\m}\left(x;\frac{1}{q}\right).$$
By comparing the coefficients of $Q_{\m}\left(x;\frac{1}{q}\right)$ in \eqref{eq 10} and \eqref{eq 24}, we get the desired result.
\end{proof}

\begin{proof}[Proof of Theorem \ref{subnormality alternative}(discrete part)]
For each fixed $0 \leq k \leq r-1,$ let $\nu=\frac{a}{2}(r-1)+b+1+k\frac{a}{2}.$
Consider the embedding $\pi_{k,r}:\Delta_k\to \Delta_r$ given by $\pi_{k,r}(t_1,\ldots,t_k)=(t_1,\ldots, t_k, 1,\ldots,1)$ and set $ p_k=a(k-1)+b+2.$ For every Borel subset $A$ of $\Delta_r$, define a measure  $\mu_r^\nu$ as \beq  \label{measure for discrt} \mu_r^\nu(A)=(\mu_k^{\nu}\circ \pi_{k,r}^{-1})(A).\eeq
Here, $\mu_k^\nu$ is the measure supported on $\Delta_k$ obtained by replacing $r$ with $k$ in \eqref{measure for cont}, that is, 
$ d\mu_k^\nu(t)=C_\nu(k)\prod_{i=1}^k t_i^b (1-t_i)^{\nu-p_k}\prod_{1 \leq i <j \leq k} |t_i-t_j|^{a}dt_1 \cdots dt_k.$ Since $\nu> p_k-1$  for all $0 \leq k \leq r-1,$ the measure $\mu_k^\nu$ is a finite measure on $\Delta_k$.

Thus, \beqn  \int_{\Delta_r}S_{\m}^{a/2} (t)d\mu_r^\nu(t)
&=& \int_{\Delta_k} S_{\m}^{a/2}(t_1,\ldots,t_k,1,\ldots,1)d\mu_{k}^\nu(t_1,\ldots,t_k)\\
&=& C_\nu(k) \int_{\Delta_k} S_{\m}^{a/2} (t_1,\ldots, t_k,1,\ldots,1) \\
&& \times \prod_{i=1}^k t_i^b (1-t_i)^{\frac{a}{2}(r-k+1)-1}\prod_{1 \leq i <j \leq k} |t_i-t_j|^{a}dt_1 \cdots dt_k.\eeqn 
In order to complete the proof, it suffices to show that \beq \label{subnormal eq3}\int_{[0,1]^k}S_{\m}^{a/2}(t_1,\ldots,t_k, 1,\ldots,1) \prod_{i=1}^k t_i^b (1-t_i)^{\frac{a}{2}(r-k+1)-1}\prod_{1 \leq i <j \leq k} |t_i-t_j|^{a}dt_1 \cdots dt_k\notag\\
= k! f_k^{a/2}(0) \prod_{i=1}^k \frac{\Gamma(b+1+\frac{a}{2}(k-i))\Gamma(\frac{a}{2}(r+1-i))}{\Gamma(\frac{a}{2}(r-i)+b+1+k\frac{a}{2})}\notag\\
\times \frac{f_r^{a/2}(\m)}{f_r^{a/2}(0)} \frac{(\frac{a}{2}(r-1)+b+1)_{\m}}{(\frac{a}{2}(r-1)+b+1+k\frac{a}{2})_{\m}}.\eeq

Since, by \cite[Equation 9.26]{KD1997},  $ S_{\m}^{a/2} (t_1,\ldots,t_k,1,\ldots,1)=\sum_{\s\subseteq \m} S_{\m/\s}^{a/2}(1^{r-k})S_{\s}^{a/2}(t_1,\ldots,t_k),$ 
    \begin{align*}
        & \int_{[0,1]^k}S_{\m}^{a/2}(t_1,\ldots,t_k, 1,\ldots,1) \prod_{i=1}^k t_i^b (1-t_i)^{\frac{a}{2}(r-k+1)-1}\prod_{1 \leq i <j \leq k} |t_i-t_j|^{a}dt_1 \cdots dt_k\\
    &= \sum_{\s\subseteq \m} S_{\m/\s}^{a/2} (1^{r-k})\int_{[0,1]^k}S_{\s}^{a/2}(t_1,\ldots, t_k) \prod_{i=1}^k t_i^b (1-t_i)^{\frac{a}{2}(r-k+1)-1}\prod_{1 \leq i <j \leq k} |t_i-t_j|^{a}dt_1 \cdots dt_k\\
    &=\sum_{\s\subseteq \m} S_{\m/\s}^{a/2} (1^{r-k}) I_k^{a/2}\left(b+1, \frac{a}{2}(r-k+1); \s\right),
    \end{align*}
    where $I_k^{a/2}\left(b+1, \frac{a}{2}(r-k+1), \s\right),$ generalized Selberg's integral \cite[Theorem 1]{KD1997}, is given by $$I_k^{a/2}\left(b+1, \frac{a}{2}(r-k+1), \s\right)=k! f_k^{a/2}(\s) \prod_{i=1}^k \frac{\Gamma\left(b+1+\frac{a}{2}(k-i)+s_i\right)\Gamma\left(\frac{a}{2}(r-k+1)+\frac{a}{2}(k-i)\right)}{\Gamma\left(b+1+\frac{a}{2}(r-k+1)+\frac{a}{2}(2k-i-1)+s_i\right)}.$$ 
    Note that $S_{\s}^{a/2}(1^k)=\frac{f_k^{a/2}(\s)}{f_k^{a/2}(0)}$ (see  \cite[Equation S]{KD1997}). Thus, we get 
    \beqn \frac{I_k^{a/2}\left(b+1, \frac{a}{2}(r-k+1); \s\right)}{I_k^{a/2}\left(b+1, \frac{a}{2}(r-k+1); 0\right)}&=& S_{\s}^{a/2}(1^k) \prod_{i=1}^k \frac{\left(b+1+\frac{a}{2}(k-i)\right)_{s_i}}{\left(b+1+\frac{a}{2}(r-k+1)+\frac{a}{2}(2k-i-1)\right)_{s_i}}\\
    &=& S_{\s}^{a/2}(1^k) \prod_{i=1}^k \frac{\left(b+1+\frac{a}{2}(k-1)-\frac{a}{2}(i-1)\right)_{s_i}}{\left(b+1+\frac{a}{2}(r+k-1)-\frac{a}{2}(i-1)\right)_{s_i}}\\
    &=& S_{\s}^{a/2}(1^k) \frac{\left(b+1+\frac{a}{2}(k-1)\right)_{\s}}{\left(b+1+\frac{a}{2}(r-1)+\frac{a}{2}k\right)_{\s}}.\eeqn
    Therefore, 
\begin{align*}
    & \int_{[0,1]^k}S_{\m}^{a/2}(t_1,\ldots,t_k, 1,\ldots,1) \prod_{i=1}^k t_i^b (1-t_i)^{\frac{a}{2}(r-k+1)-1}\prod_{1 \leq i <j \leq k} |t_i-t_j|^{a}dt_1 \cdots dt_k\\
    &= I_k^{a/2}\left(b+1, \frac{a}{2}(r-k+1); 0\right) \sum_{\s\subseteq \m} S_{\m/\s}^{a/2}(1^{r-k}) S_{\s}^{a/2}(1^k)  \frac{\left(b+1+\frac{a}{2}(k-1)\right)_{\s}}{\left(b+1+\frac{a}{2}(r-1)+\frac{a}{2}k\right)_{\s}}\\
    &{=} I_k^{a/2}\left(b+1, \frac{a}{2}(r-k+1); 0\right) S_{\m}^{a/2}(1^r)\frac{\left(b+1+\frac{a}{2}(k-1)+\frac{a}{2}(r-k)\right)_{\s}}{\left(b+1+\frac{a}{2}(r-1)+\frac{a}{2}k\right)_{\s}} \;\;\text{(Proposition \ref{important lemma})}\\
   &= I_k^{a/2}\left(b+1, \frac{a}{2}(r-k+1); 0\right) S_{\m}^{a/2}(1^r)\frac{\left(b+1+\frac{a}{2}(r-1)\right)_{\s}}{\left(b+1+\frac{a}{2}(r-1)+\frac{a}{2}k\right)_{\s}} \\
   &=k! f_k^{a/2}(0) \prod_{i=1}^k \frac{\Gamma\left(b+1+\frac{a}{2}(k-i)\right)\Gamma\left(\frac{a}{2}(r+1-i)\right)}{\Gamma\left(\frac{a}{2}(r-i)+b+1+\frac{a}{2}k\right)}\frac{f_r^{a/2}(\m)}{f_r^{a/2}(0)} \frac{\left(\frac{a}{2}(r-1)+b+1\right)_{\m}}{\left(\frac{a}{2}(r-1)+b+1+\frac{a}{2}k\right)_{\m}}.
\end{align*}
This completes the proof.
 \end{proof}   
 The following result is an immediate consequence of Theorem \ref{contractionpoint} and Theorem \ref{subnormality alternative}.
\begin{corollary}
    Let $\boldsymbol M_z^{(\nu)}$ be the multiplication operator tuple on $\mathcal{H}^\nu(\Omega).$ Then
    $\boldsymbol M_z^{(\nu)}$ is subnormal if and only if it is a contractive $d$-tuple.
\end{corollary}
Let $m \geq 2,$ and $\Omega=\{z \in \mathbb C^{2 \times m}: \|z\| <1\}$ be a type-I domain of rank $2.$ The characteristic multiplicities are $a=2$ and $b=m-2.$ For $ \nu, \eta >1,$ let $K^{\nu,\eta}(z,w):\Omega \times \Omega\to \mathbb C$ be a sesqui-analytic function given by $$K^{\nu,\eta}(z,w)=\sum_{\s} \frac{(\nu)_{\s}}{(\eta)_{\s}}(m)_{\s} K_{\s}(z,w),\; z, w\in \Omega.$$ 
By \cite[Lemma 5.1]{FK1990}, $K^{\nu,\eta}$ is positive semi-definite. Let $\boldsymbol M_z^{\nu,\eta}$ denote the tuple of multiplication operators by the coordinate functions on the reproducing kernel Hilbert space $\mathcal{H}^{\nu,\eta}(\Omega)$ determined by the kernel $K^{\nu,\eta}.$  It follows from \cite[Theorem 3.5]{GKP2022} that $\boldsymbol M_z^{\nu,\eta}$ is bounded.
Indeed, $$\sup\left\{ \frac{(\nu)_{\s-\epsilon_i}}{(\nu)_{\s}}\frac{(\eta)_{\s}}{(\eta)_{\s-\epsilon_i}}: \s, \s-\epsilon_i \in \Vec{\mathbb N^2}, i=1,2\right\}\leq \max\left\{\frac{\eta}{\nu}, \frac{\eta-1}{\nu-1}\right\}.$$ 
%$\frac{f_2^1(\s)}{f_2^1(\underline 0)}\frac{(\eta)_{\s}}{(\nu)_{\s} }$
As an application of Corollary \ref{complete characterization}, we obtain the subnormality of $\boldsymbol M_z^{\nu, \eta}$ in the following result.

 \begin{proposition}
      $\boldsymbol M_z^{\nu, \eta}$ is subnormal contractive $2m$-tuple if and only if $\nu=\eta$ or $\nu\geq \eta+1.$ 
     Moreover, the representing measure for the sequence $\{c(s)\}_{s\in \mathbb Z_+^2}$ is explicitly given by the following cases:
     \begin{enumerate}
         \item [(i)] For $\nu=\eta$, the measure supported on $[0,1]^2$ is given by  $d\mu^\nu(t)=2\;d\delta (t_1)\otimes d\delta(t_2),$ where $\delta(t_i)$ is the Dirac delta measure on the singleton set $\{1\}.$
         \item [(ii)] For $\nu=\eta+1$, the measure supported on $[0,1]^2$ is given by $d\mu^\nu(t)=\eta(\eta-1)[t_1^{\eta-2}(1-t_1) dt_1 \otimes d\delta (t_2)+d\delta(t_1)\otimes t_2^{\eta-2}(1-t_2)dt_2].$
         \item [(iii)] For $\nu>\eta+1$, the measure supported on $[0,1]^2$ is given by $$d\mu^\nu(t)=\prod_{i=1}^2 \frac{\Gamma(\nu-i+1)}{\Gamma(\eta-i+1)\Gamma(\nu-\eta-i+1)} t_1^{\eta-2}t_2^{\eta-2}(1-t_1)^{\nu-\eta-2}(1-t_2)^{\nu-\eta-2} (t_1-t_2)^2dt_1 dt_2.$$
     \end{enumerate}
 \end{proposition}
\begin{proof} 
For the rank $2$ type-I domain, the entries $K_{\m,\s}$ of the Kostka matrix are $1$ for all $\s\leq\m$, and $K_{\m,\s}=0$ otherwise (see \eqref{rank 2 Kostka}). Consequently, the inverse Kostka matrix has $1$'s on the diagonal, $-1$'s on the principal superdiagonal, and zeros elsewhere. Thus, using \eqref{relation between two sequence}, we obtain the sequence $\{c(s)\}_{s\in \mathbb Z_+^2}$ as follows:
\beqn c(s)=\frac{\Gamma(\eta+s_1-1)\Gamma(\eta+s_2-1)\Gamma(\nu)\Gamma(\nu-1)}{\Gamma(\eta)\Gamma(\eta-1)\Gamma(\nu+s_1-1)\Gamma(\nu+s_2-1)}\left[2+(\eta-\nu)\frac{2\nu+s_1+s_2-2-(s_1-s_2)^2}{(\nu+s_1-1)(\nu+s_2-1)}\right],\eeqn for every $s=(s_1,s_2)\in \mathbb Z_+^2.$
In view of Corollary \ref{complete characterization}, the subnormality of $\boldsymbol{M}_z^{\nu, \eta}$ is equivalent to proving that the sequence $\{c(s)\}_{s\in \mathbb Z_+^2}$ is a Hausdorff moment sequence.
It is known that a (multi) sequence is a Hausdorff moment sequence if and only if it is completely monotone (see \cite[Proposition 6.11]{BCR1984}).
 If $\{c(s)\}_{s\in \mathbb Z_+^2}$ is completely monotone, then necessarily $c(0,0)-c(1,0)-c(0,1)+c(1,1)= 2-2\frac{\eta}{\nu}-2\frac{\eta}{\nu}-2\frac{\eta(\eta-1)}{\nu(\nu-1)}=2\frac{(\nu-\eta)(\nu-\eta-1)}{\nu(\nu-1)}\geq 0.$ Since $\nu>1, \; \nu \in \{\eta\}\cup [\eta+1,\infty).$  
 
 Conversely, if $\nu=\eta$, then $c(s)=2$ for all $s\in \mathbb Z_+^2.$ Therefore, $\{c(s)\}$ is completely monotone.
For $\nu=\eta+1$, it is evident that $$c(s_1,s_2)=\eta(\eta-1)\left[\frac{1}{(\eta+s_2)(\eta+s_2-1)}+\frac{1}{(\eta+s_1)(\eta+s_1-1)}\right].$$ For any positive integer $n,$ if $f(n)=\frac{1}{(\eta+n)(\eta+n-1)}$ and $(\Delta f)(n)=f(n+1)-f(n)$,
then $$(-1)^k(\Delta^k f)(n)=\frac{ (k+1)!}{(\eta+n-1)(\eta+n)\cdots (\eta+n+k)} \geq 0, \; \text{ for all }\; n, k \in \mathbb Z_+.$$ This shows that $\{c(s)\}$ is completely monotone for $\nu=\eta+1.$ 
 
The case when $\nu>\eta+1,$ consider $ \beta=\nu-\eta-1>0,$ and for $s=(s_1,s_2)\in \mathbb Z_+^2,$ $$H(s_1,s_2)=\frac{\Gamma(\eta+s_1-1)\Gamma(\eta+s_2-1)}{\Gamma(\nu+s_1-1)\Gamma(\nu+s_2-1)}\left[2+(\eta-\nu)\frac{2\nu+s_1+s_2-2-(s_1-s_2)^2}{(\nu+s_1-1)(\nu+s_2-1)}\right].$$ 
 Note that the complete monotonicity of $\{c(\s)\}$ is equivalent to the completely monotonicity of $\{H(s)\}_{s\in \mathbb Z_+^2}$. Let $g(n)=\frac{\Gamma(\eta+n-1)}{\Gamma(\nu+n-2)},$ then for any positive integer $k$, $$U_k(n):=(-1)^k(\Delta^k g)(n)= \frac{\Gamma(\beta+k)}{\Gamma(\beta)}\frac{\Gamma(\eta+n-1)}{\Gamma(\nu+n-2+k)}.$$ 
Clearly, $U_k(n) \geq 0$ whenever $\nu >\eta+1.$  For any function $h:\mathbb Z_+^2 \to \mathbb C$, we consider the maps $E_1$ and $E_2$ given by $E_1(h(s_1,s_2)):= h(s_1+1,s_2)$ and $E_2(h(s_1,s_2)):=h(s_1,s_2+1).$
 It is easy to check that $$H(s_1,s_2)=\frac{1}{\beta} (E_1-E_2)^2 (g(s_1)g(s_2)).$$ Here, by the product $E_iE_j$, we mean the composition of the corresponding maps.
 Therefore, for any positive integers $k_1,k_2$,
 \begin{align*}
     &(-1)^{k_1+k_2} \Delta_1^{k_1}\Delta_2^{k_2}H(s_1,s_2) = (-1)^{k_1+k_2} \Delta_1^{k_1}\Delta_2^{k_2} \left(\frac{1}{\beta} (E_1-E_2)^2 (g(s_1)g(s_2))\right)\\
 &= \frac{1}{\beta} (E_1-E_2)^2 \left[ (-1)^{k_1}(\Delta_1^{k_1} g)(s_1) (-1)^{k_2} (\Delta_2^{k_2}g)(s_2)\right]=\frac{1}{\beta} (E_1-E_2)^2 (U_{k_1}(s_1)U_{k_2}(s_2))\\
 &= \frac{1}{\beta}\left[ U_{k_1}(s_1+2)U_{k_2}(s_2)-2U_{k_1}(s_1+1)U_{k_2}(s_2+1)+U_{k_1}(s_1)U_{k_2}(s_2+2)\right]\\
 &= \frac{U_{k_1}(s_1)U_{k_2}(s_2)}{\beta}\left[\frac{U_{k_1}(s_1+2)}{U_{k_1}1(s_1)}-2\frac{U_{k_1}(s_1+1) U_{k_2}(s_2+1)}{U_{k_1}(s_1)U_{k_2}(s_2)}+\frac{U_{k_2}(s_2+2)}{U_{k_2}(s_2)}\right]\\
 &\geq \frac{U_{k_1}(s_1)U_{k_2}(s_2)}{\beta} \left[\left(\frac{U_{k_1}(s_1+1)}{U_{k_1}(s_1)}\right)^2-2\frac{U_{k_1}(s_1+1)}{U_{k_1}(s_1)}\frac{U_{k_2}(s_2+1)}{U_{k_2}(s_2)}+\left(\frac{U_{k_2}(s_2+1)}{U_{k_2}(s_2)}\right)^2 \right] \geq 0.
 \end{align*}
 In the second last inequality, we have used that $\frac{U_k(n+2)}{U_k(n+1)}>\frac{U_k(n+1)}{U_k(n)},$ whenever $\nu>\eta+1.$ 
Therefore, the sequence $\{H(s_1,s_2)\}$ is completely monotone. The moreover part of the result is a routine verification. 
\end{proof}


\begin{thebibliography}{33}

\bibitem{AJ1995} 
J. Arazy, \emph{A survey of invariant {H}ilbert spaces of analytic functions on bounded symmetric domains}, Contemporary Mathematics, {\bf 185}  (1995), 7-65.

\bibitem{AZ2003}
J. Arazy and G. Zhang, \emph{Homogeneous multiplication operators on bounded symmetric domains}, J. Funct. Anal., {\bf 202} (2003), no. 1, 44–66.

\bibitem{At1990}
A. Athavale, \emph{On the intertwining of joint isometries}, J. Operator Theory, {\bf 23} (1990), no. 2, 339–350.

\bibitem{BM1996}
B. Bagchi and G. Misra, \emph{Homogeneous tuples of multiplication operators on twisted Bergman spaces}, J. Funct. Anal., {\bf 136 } (1996), no. 1, 171–213.

\bibitem{BCR1984}
C. Berg, J. P. Christensen, and P. Ressel, \emph{Harmonic analysis on semigroups}, Grad. Texts in Math., 100
Springer-Verlag, New York, 1984. x+289 pp.

\bibitem{Ca1935}
\'E. Cartan, \emph{Sur les domaines born\'es homog\'enes de l'espace den variables complexes}, Abh. Math. Semin. Univ. Hambg., {\bf 11} (1935), 116-162.

\bibitem{CY2015}
S. Chavan and D. Yakubovich,
\emph{Spherical tuples of Hilbert space operators},
Indiana Univ. Math. J., {\bf 64} (2015), 577-612.


\bibitem{Cn1991}
 J. B. Conway, \emph{The theory of subnormal operators}, Math. Surveys Monogr., {\bf 36}, American Mathematical Society, Providence, RI, 1991. xvi+436 pp.

 \bibitem{Cn19912}
 J. B. Conway, \emph{Towards a functional calculus for subnormal tuples: the minimal normal extension}, Trans. Amer. Math. Soc., {\bf 326} (1991), no. 2, 543–567.

\bibitem{RC1985}
R. E. Curto, N. Salinas, \emph{Spectral properties of cyclic subnormal $m$-tuples}, Amer. J. Math., {\bf 107} (1985), no. 1, 113–138.

\bibitem{RC1988}
R. E. Curto, \emph{Applications of several complex variables to multiparameter spectral
theory} In Surveys of some recent results in operator theory, Vol. II, volume 192 of
Pitman Res. Notes Math. Ser., pages 25–90. Longman Sci. Tech., Harlow, 1988.

\bibitem{DE2005}
M. Didas and J. Eschmeier, \emph{Subnormal tuples on strictly pseudoconvex and bounded symmetric domains}, Acta Sci. Math. (Szeged), {\bf 71} (2005), no. 3-4, 691–731.

\bibitem{HD2003}
 H. Duan, \emph{On the inverse Kostka matrix}, J. Combin. Theory Ser. A, {\bf 103} (2003), no. 2, 363-376.

 \bibitem{ER1990}
O. Egecioglu, J. B. Remmel, \emph{A combinatorial interpretation of the inverse Kostka matrix}, {Linear and Multilinear Algebra}, {\bf 26} (1990), no. 1-2, 59-84.



\bibitem{FK1990}
J. Faraut and A. Koranyi, \emph{Function spaces and reproducing kernels on bounded symmetric domains}, J. Funct. Anal., {\bf 88} (1990), 64-89.

\bibitem{GKP2022}
S. Ghara, S. Kumar, and P. Pramanick, \emph {$\mathbb K$-homogeneous tuple of operators on bounded symmetric domains}, Israel J. Math., {\bf 247} (2022), 331-360.

\bibitem{GR2006}
J. Gleason and S. Richter, \emph{$m$-Isometric Commuting Tuples of Operators on a Hilbert Space}, Integral Equations Operator Theory, {\bf 56} (2006), 181–196.

\bibitem{KD1988}
K. W. J. Kadell, \emph{A proof of some $q$-analogues of Selberg's integral for $k=1$}, SIAM J. Math. Ana., {\bf 19}(1988), 944-968.

\bibitem{KD1997}
K. W. J. Kadell, \emph{The Selberg-Jack symmetric functions}, Adv. Math., {\bf 130} (1997), no. 1, 33–102.

\bibitem{KM2019}
A. Kor\'anyi and G. Misra, \emph{Homogeneous Hermitian holomorphic vector bundles and the Cowen-Douglas class over bounded symmetric domains}, Adv. Math., {\bf 351} (2019), 1105–1138.

\bibitem{KMP2025}
 S. Kumar, M. K. Mal, and P. Pramanick, \emph{Cartan isometries and Toeplitz operators on Cartan domains}, 
https://doi.org/10.48550/arXiv.2505.24325.

\bibitem{Ls1977}
O. Loos, \emph{Bounded symmetric domains and Jordan pairs}, University of California, Irvine, 1977.

\bibitem{MD1995}
I. G. Macdonald, \emph{Symmetric Functions and Hall Polynomials}, 2nd ed., Oxford
Univ. Press, Oxford, 1995.

\bibitem{MS1990}
G. Misra and N. S. N. Sastry, \emph{Homogeneous tuples of operators and representations of some classical groups},  J. Operator Theory, {\bf 24} (1990), no. 1, 23–32.

\bibitem{MU2016}
G. Misra and H. Upmeier, \emph{Homogeneous vector bundles and intertwining operators for symmetric domains}, Adv. Math., {\bf 303} (2016), 1077–1121.

\bibitem{SZ1974}
Z. Slodkowski and W. \.Zelazko, \emph{On joint spectra of commuting families of operators}, Studia Math., {\bf 50} (1974), 127-148.

\bibitem{St1989}
R. P. Stanley, \emph{Some combinatorial properties of Jack symmetric functions}, Adv. Math., {\bf 77} (1989), no. 1, 76–115.


\bibitem{UP1996}
H. Upmeier, \emph{Toeplitz operators and index theory in several complex variables}, Operator Theory Advances and Applications,  {\bf 81}, Birkh\"auser, 1996.

\bibitem{Up2021}
 H. Upmeier, \emph{Eigenvalues of $K$-invariant Toeplitz operators on bounded symmetric domains}, Integral Equations Operator Theory, {\bf 93}, (2021), Article no. 27. 

\bibitem{HU2023}
H. Upmeier, \emph{$K$ -invariant Hilbert modules and singular vector bundles on bounded symmetric domains}, J. Reine Angew. Math., {\bf 799} (2023), 155–187.


\bibitem{YZ1992} Z. Yan, \emph{A class of generalized hypergeometric functions in several variables},  Canad. J. Math., {\bf 44} (1992), no. 6, 1317–1338.


\end{thebibliography}
\end{document}